\documentclass[11pt, reqno, a4paper]{amsart}
\usepackage{amssymb,latexsym,amsmath,amsfonts, amsthm}
\usepackage{latexsym}
\usepackage[numbers,sort&compress]{natbib}
\usepackage{color}
\usepackage[mathscr]{eucal}
\usepackage{comment}
\usepackage{dsfont}
\usepackage[top=2.7cm,bottom=2.7cm,left=2.7cm,right=2.7cm]{geometry}
\makeatletter
\def\section{\@startsection{section}{1}%
	\z@{.7\linespacing\@plus\linespacing}{.5\linespacing}%
	{\bfseries
		\centering
}}
\def\@secnumfont{\bfseries}
\makeatother

\usepackage{amsmath,amssymb,amsthm,graphicx,amsxtra, setspace}
\usepackage[utf8]{inputenc}
\usepackage{mathrsfs}
\usepackage{alltt}
\usepackage{relsize}
\usepackage{hyperref}
\usepackage{aliascnt}
\usepackage{tikz}
\usepackage{mathtools}
\usepackage{multicol}
\usepackage{upgreek}
\usepackage{graphicx,type1cm,eso-pic,color}

\allowdisplaybreaks

\numberwithin{equation}{section}

\usepackage[cyr]{aeguill}

\colorlet{darkblue}{blue!50!black}

\hypersetup{colorlinks=true,%
	citecolor=red,%
	filecolor=blue,%
	linkcolor=darkblue,%
}

\newtheorem{theorem}{Theorem}[section]

\newaliascnt{lemma}{theorem}
\newtheorem{lemma}[lemma]{Lemma}
\aliascntresetthe{lemma} 

\newaliascnt{proposition}{theorem}

\aliascntresetthe{proposition}

\newaliascnt{assumption}{theorem}
\newtheorem{assumption}[assumption]{Assumption}
\aliascntresetthe{assumption}

\newaliascnt{corollary}{theorem}
\newtheorem{corollary}[corollary]{Corollary}
\aliascntresetthe{corollary}

\newaliascnt{definition}{theorem}
\newtheorem{definition}[definition]{Definition}
\aliascntresetthe{definition}

\newaliascnt{example}{theorem}

\aliascntresetthe{example}

\newaliascnt{remark}{theorem}
\newtheorem{remark}[remark]{Remark}
\aliascntresetthe{remark}

\newaliascnt{hypothesis}{theorem}

\aliascntresetthe{hypothesis}

\newaliascnt{property}{theorem}

\aliascntresetthe{property}

\let\originalleft\left
\let\originalright\right
\renewcommand{\left}{\mathopen{}\mathclose\bgroup\originalleft}
\renewcommand{\right}{\aftergroup\egroup\originalright}

\makeatletter
\newcommand{\doublewidetilde}[1]{{%
		\mathpalette\double@widetilde{#1}%
}}
\newcommand{\double@widetilde}[2]{%
	\sbox\z@{$\m@th#1\widetilde{#2}$}%
	\ht\z@=.9\ht\z@
	\widetilde{\box\z@}%
}
\makeatother

\DeclareMathOperator*{\esssup}{ess\,sup}

\renewcommand{\d}{\/\mathrm{d}\/}

\def\w{\textbf{W}^{\varepsilon}_{{\theta}^{\varepsilon}}}

\def\L{\mathrm{L}}
\def\A{\mathrm{A}}

\def\C{\mathrm{C}}

\def\B{\mathrm{B}}
\def\D{\mathrm{D}}

\def\X{\mathbb{X}}
\def\x{\boldsymbol{x}}
\def\z{\boldsymbol{z}}
\def\v{\boldsymbol{v}}
\def\V{\mathbb{V}}
\def\w{\boldsymbol{w}}

\def\N{\mathbb{N}}

\def\no{\nonumber}
\def\V{\mathbb{V}}

\def\U{\mathrm{U}}
\def\P{\mathbb{P}}
\def\u{\boldsymbol{u}}
\def\H{\mathbb{H}}
\def\n{\boldsymbol{n}}

\newcommand{\R}{\mathbb{R}}

\renewcommand{\d}{\/\mathrm{d}\/}

\newcommand{\Addresses}{{
		\footnote{
				\footnotesize

			\noindent \textsuperscript{1,2}School of Mathematics,
			Indian Institute of Science Education and Research, Thiruvananthapuram (IISER-TVM),
			Maruthamala PO, Vithura, Thiruvananthapuram, Kerala, 695 551, INDIA.  \par\nopagebreak \noindent
			 \textit{e-mail:}  \texttt{Manika Bag: manikabag19@iisertvm.ac.in}		\\
		 \textit{e-mail:} \texttt{Sheetal Dharmatti: sheetal@iisertvm.ac.in }		
			
				\noindent \textsuperscript{3}Department of Mathematics, Indian Institute of Technology Roorkee-IIT Roorkee,
			Haridwar Highway, Roorkee, Uttarakhand 247667, INDIA.\par\nopagebreak
			\noindent  \textit{e-mail:} \texttt{Manil T. Mohan: maniltmohan@ma.iitr.ac.in, maniltmohan@gmail.com.}

			\noindent \textsuperscript{*}Corresponding author.
			
				\textit{Keywords: Cahn-Hilliard equation, Brinkman-Frochheimer equations,  Leray-Hopf weak solution, energy equality} 
			
			Mathematics Subject Classification (2020): Primary 35A01, 35A02, 76B03 ; Secondary 35Q35, 76D03 .

			\medskip\noindent
			{\bf Acknowledgments:} 
			Manika Bag would like to thank the  Indian Institute of Science Education and Research, Thiruvananthapuram, for providing financial support and stimulating the environment for the research.  Sheetal Dharmatti would like to thank  Science $\&$ Engineering Research Board (SERB), India, Core Research grant, CRG/2021/008278.
			M. T. Mohan would  like to thank the Department of Science and Technology (DST) Science $\&$ Engineering Research Board (SERB), India for a MATRICS grant (MTR/2021/000066). 
			
}}}
\begin{document}

\title[Well-posedness of a 3D Damped Cahn-Hilliard-Navier-Stokes Equations]{Well-posedness of three-dimensional Damped Cahn-Hilliard-Navier-Stokes Equations	\Addresses	}
	
	\author[M. Bag, S. Dharmatti and  M. T. Mohan]
	{Manika Bag\textsuperscript{1}, Sheetal Dharmatti\textsuperscript{2*} and Manil T. Mohan\textsuperscript{3}} 
\maketitle
	
\begin{abstract} 
 This paper presents a mathematical analysis of the evolution of a mixture of two incompressible, isothermal fluids flowing through a porous medium in a three-dimensional bounded domain. The model is governed by a coupled system of convective Brinkman–Forchheimer equations and the Cahn–Hilliard equation, considering a regular potential and non-degenerate mobility. We first establish the existence of a Leray–Hopf weak solution for the coupled system when the absorption exponent 
$r\geq1$. Additionally, we prove that every weak solution satisfies the energy equality for 
$r\geq3$. This further leads to the uniqueness of weak solutions in three-dimensional bounded domains, subject to certain restrictions on the \emph{viscosity} ($\nu$) and the \emph{Forchheimer coefficient} ($\beta$) in the critical case 
$r=3$. Moreover, we provide an alternative simplified proof for the uniqueness of weak solutions for 
$r\geq 3$ that holds without imposing any restrictions on $\nu$ or $\beta$. Similar results are also obtained for the case of degenerate mobility and singular potential.  
\end{abstract}

\section{Introduction}
We study here a diffuse interface model characterized by the local Cahn-Hilliard equation coupled with the convective Brinkman-Frochheimer equations or damped Navier-Stokes equations. To the best of authors' knowledge, this is the first paper considering such a coupling and its mathematical analysis. Our aim is to study the well-posedness of this  model in three dimensions (3D) with various possible choices for the potential and mobility.

The diffuse interface theory explains how two fluids which do not squish easily but flow thickly move and how the boundary between them changes. This formulation was first introduced by Hohenberg and Halperin in \cite{H}. The model is based on the equilibrium between mass and momentum, incorporating constitutive laws that align with a variant of the second law of thermodynamics. The model has been utilised in various numerical investigations for specific practical applications, such as droplet formation and collision, moving contact lines, large deformation flows, etc. We refer the readers to \cite{diffuse} and citations contained therein for an overview of these topics. The above formulation is characterised by the Cahn-Hilliard equation coupled  with Navier-Stokes equations.   If we denote the relative concentration  of the fluid by $\varphi$ and the average velocity by $\boldsymbol{u}$; in a  simplified setting when density is  constant; the above characterisation reduces to
\begin{equation}\label{equ Q}
\left\{
\begin{aligned}
  &\varphi' + \boldsymbol{u} \cdot \nabla \varphi = \mathrm{div}(m(\varphi)\nabla \mu),\\
 &   \boldsymbol{u}' - 2\mathrm{div}(\nu(\varphi) D\boldsymbol{u}) + (\boldsymbol{u}\cdot \nabla)\boldsymbol{u}  + \nabla \pi = \mu \nabla \varphi,\\
& \mathrm{div}~\boldsymbol{u}  = 0,
\end{aligned}   
\right.
\end{equation}
in $ \Omega \times (0,T)$, where $\u(x,t)= (u_1, u_2, u_3)$ is the velocity field and the scalar fields $\varphi$ and $\pi$ denote the relative concentration and pressure of the fluid, respectively and $\Omega$ is a bounded domain in $\mathbb{R}^3$ with sufficiently smooth boundary $\partial\Omega$. The chemical potential $\mu$ is the variation of the free energy functional $$ \mathcal{E}(\varphi)=\int_{\Omega}\Big(\frac{1}{2}|\nabla\varphi(x)|^2+ F(\varphi(x))\Big) dx,$$  where $F$ represents the double-well potential whose wells are located in pure phases \cite{ha13, bosia_pullback_attractor, 1996_ch, boyer1999mathematical}. Further generalizations of the model have been made by choosing  $F$ to be a smooth potential (defined on the whole real line) or singular potential (defined on a bounded interval)  and by considering mobilities $m (\varphi) $ of different types  namely  non-degenerate or degenerate \cite{boyer1999mathematical, ha13, ha09diffuse, uniqueness, gal10} . By considering different forms of the free energy, other generalizations of the model are also studied \cite{giorgini2021navier, colli_weak, ha14, nonloccahn, 2015Frigeri}. For example, if the gradient term is replaced by a nonlocal spatial operator, one can define $$ E(\varphi)=\frac{1}{4}\int_{\Omega}\int_{\Omega}J(x-y)(\varphi(x)-\varphi(y))^2 dx \, dy + \int_{\Omega}F(\varphi(x)) dx,$$ 
and the chemical potential associated with the nonlocal model  given by  its first variation namely $$\mu= a\varphi-J\ast\varphi+F'(\varphi), \ \text{ where }\  a(x):=\int_{\Omega}J(x-y) dy, \ x\in\Omega,$$ where $J:\mathbb{R}^3\to\mathbb{R}$ is a smooth even function.  The system \eqref{equ Q} associated with this $\mu$ is called nonlocal Cahn-Hilliard -Navier- Stokes model (nonlocal CHNS model). A brief review of this diffuse interface model can be found in \cite{diffuse}.

\subsection{\textbf{Literature Survey}}
We discuss briefly the results available in  the literature for CHNS model  with  different types of mobility and/or potential. Let us first consider the local CHNS model for which the mathematical analysis was presented by F. Boyer in \cite{boyer1999mathematical} for the case when $\Omega\subset\mathbb{R}^d,$ $d\in\{2,3\}$ is a periodical channel. He showed existence of a global weak solution  in  dimensions $2$ and $3$, which is shown to be strong and unique  for $t \in (0, T_0)$ with $T_0$ sufficiently small. Moreover, the case  of singular potential  with the degenerate mobility is also considered.  A more complete mathematical theory of existence, uniqueness, regularity and asymptotic behaviour in the case of singular potential with degenerate mobility is studied by H. Abels in \cite{ha09diffuse}.  The author has  further studied fluids of different densities which lead to a non-homogeneous CHNS system,  where the density of mixture depends on fluid concentration and velocity field is no longer  divergence free. The detailed analysis of this system is studied in \cite{ha09}. The asymptotic behaviour of 2D CHNS system is studied by Gal and Grasseli in \cite{ gal10}, where the authors showed that the system generates a strongly continuous semigroup which possesses a global attractor. They  established the existence of an exponential attractor and finally proved that each trajectory converges to a single equilibrium. 
However, the  uniqueness of weak solutions in the case of singular potential and degenerate mobility, and regularity of weak solutions was  open even in two dimensions until Giorgini, Miranville, Temam proved it in \cite{uniqueness}. These questions  still  remain as challenging open problems in three  dimensions. For more details, we refer the readers  to \cite{ uniqueness, ha13, ha14}.

A nonlocal CHNS model is considered in \cite{ colli_weak} with regular potential and constant mobility, where the authors studied existence of global weak solutions. Whereas in  \cite{2015Frigeri}, the authors consider non-constant mobility case with  singular potential and degenerate mobility for the existence of weak solution, uniqueness and energy identity in a bounded domain in two dimensions. Later, the existence of strong solutions in two dimensions was established  in \cite{ jdeFrigeri}. However, the problems like  energy identity, uniqueness of weak solutions and the existence of global strong solutions in three dimensions are still open.  

In the present work, we focus on a coupled system  of the Cahn-Hilliard equation and the convective Brinkman-Frochheimer equations (\cite{Hajduk+Robinson_2017,Kinra+Mohan_2024}) with constant viscosity assuming the density difference is negligible  in a bounded domain in $\R^3$. The system in dimension less form reads as:
\begin{equation}\label{equ P}
\left\{
\begin{aligned}
  &\varphi' + \boldsymbol{u} \cdot \nabla \varphi = \mathrm{div}(m(\varphi)\nabla \mu), \, \, \text{ in } \Omega \times (0,T), \\
    &    \mu = -\Delta\varphi + F'(\varphi), \\
      &  \boldsymbol{u}' - 2\mathrm{div}(\nu(\varphi) D\boldsymbol{u}) + (\boldsymbol{u}\cdot \nabla)\boldsymbol{u} + \beta|\u|^{r-1}\u + \nabla \pi = \mu \nabla \varphi + \mathbb{U}, \, \, \text{ in } \Omega \times (0,T), \\
     &   \mathrm{div}~\boldsymbol{u}  = 0, \, \, \text{ in } \Omega \times [0,T], \\
      &  \frac{\partial\varphi}{\partial\n} = 0, \,  \frac{\partial\mu}{\partial\boldsymbol{n}} = 0, \,\, \text{ on } \partial\Omega\times[0,T], \\
   &     \u  = \boldsymbol{0}, \,\, \text{ on } \partial\Omega\times[0,T], \\
      &  \boldsymbol{u}(0) = \boldsymbol{u}_0 ,\,\, \varphi(0) = \varphi_0, \,\, \text{ in } \Omega,  
\end{aligned}   
\right.
\end{equation}
where  $\nu>0$ is the \emph{effective viscosity} of the fluids (for simplicity, we have taken it to be constant in this article), and  $\beta$ is the \emph{Forchheimer coefficient} depending on the porosity of the material. The parameter $r$ takes values in $[1,\infty)$ and is known as the \emph{absorption exponent} and $\mathbb{U}$ is a given external force.  Note that $r=3$ is known as the \emph{critical exponent} (\cite{Hajduk+Robinson_2017}). The \emph{chemical potential}  is given by $\mu$ and the \emph{mobility} by $m$ both of  which may depend on $\varphi$. The outward unit normal to the boundary $\partial\Omega$  is denoted by $\n$ and $D\u = \frac{\nabla\u+(\nabla\u)^\top}{2}$ is the symmetric gradient vector. 

 The Cahn-Hilliard (CH) equation describes the phase separation in binary alloys when it is cooled down sufficiently; leading to a partial or total nucleation, known as \emph{spinoidal decomposition}. In \eqref{equ P}, the first equation represents  CH part and $F$ is the well-known \emph{double-well potential}, which can be of the \emph{regular} or \emph{singular} type. A typical example of regular $F$ is 
\begin{align}\label{regular}
    F_{\mathrm{reg}}(s)=(s^2-1)^2, \ s \in \mathbb{R}.
    \end{align}
    A thermodynamically relevant $F$ is the following \emph{logarithmic potential}:
\begin{align}\label{log pot}
    & F_{\log}(s) =\frac{\theta}{2} \big((1+s)\log(1+s)+(1-s)\log(1-s)\big) + \frac{\theta_c}{2}(1-s^2) \quad s\in (-1,1), \quad 0 < \theta < \theta_c,
\end{align}
where $\theta$ and $\theta_c$ are proportional to absolute and critical temperatures, respectively. When absolute temperature is close to the critical temperature, we can approximate $F_{\log}$ by $F_{\mathrm{reg}}$. Moreover, the mobility $m$ can be \emph{nondegenerate} or \emph{degenerate} at $1, -1$. A typical example of $m$ which degenerate at $1, -1$ is 
\begin{align}\label{ex m}
m(s) = k(s)(1-s^2)^n,
\end{align}
 where $n\in\mathbb{N},$ $k: [-1, 1] \rightarrow \mathbb{R}$ is a $C^1$-function with $k_1 \leq k(s) \leq k_2, \ \text{ for all }\  s \in [-1, 1]$ for some positive constants $k_1$ and $ k_2.$
A physically realistic version of the Cahn-Hilliard equation is characterized by nonlocal free energy. A nonlocal Cahn-Hilliard model has been considered in \cite{nonloccahn, ha15, bates05}. The difference between the local and nonlocal Cahn-Hilliard model is the choice of the interaction potential.  We enclose the CH equation with the Neumann boundary condition $\frac{\partial\mu}{\partial\boldsymbol{n}}=0$ with a natural variational boundary condition $\frac{\partial\varphi}{\partial\boldsymbol{n}}=0$. Because of these boundary conditions, we have \emph{conservation of mass}, that is, the spatial average of the order parameter obtained by formal integration of the equation $\eqref{equ P}_1$ over $\Omega$,
\begin{align}\label{av}
\langle\varphi(t)\rangle = \frac{1}{|\Omega|}\int_\Omega\varphi(x,t) dx = \langle\varphi(0)\rangle, \ \text{ for all }\  t\geq 0.
\end{align}
For the sake of completeness, we refer the interested readers to \cite{CH13signdiff,1996_ch,long20ch,miranville21cahn} for a diverse analysis of the CH equation.
The equation $\eqref{equ P}_3$ is called the convective Brinkman-Frochheimer equation which describes the motion of incompressible fluid flow through a porous medium. The model has real-life applications, such as in non-Newtonian fluid flow, tidal dynamics etc. (see \cite{Hajduk+Robinson_2017,Tamed_ns, Kinra+Mohan_2024,MTMohan2,Varga} and references therein). The  equation $\eqref{equ P}_3$ with $\mu=0$ is sometimes referred as the \emph{damped Navier-Stokes equation} due to the  presence  of the  damping or absorption term  $|\u|^{r-1}\u$. The two equations  are coupled with the non linear coupling terms $ \mathrm{div}(m(\varphi)\nabla \mu)$ and  $\mu \nabla \varphi $ in the  concentration  and velocity equations, respectively.
\subsection{Motivation for the model, difficulties and novelties}
The system \eqref{equ P} describes the evolution of an isothermal mixture of two incompressible and immiscible fluids through a saturated porous medium with (relative) concentration $\varphi$ 
and the (averaged) velocity field $\u$. The flow described by this system behaves like the Cahn-Hilliard-Navier-Stokes (CHNS) fluid flows. The CHNS system describes many physically important  phenomena in fluid dynamics and material sciences and hence the model is extensively studied in the literature. However, as in the case of 3D Navier-Stokes' equations, there are several open problems related to 3D CHNS system. Mainly, the questions concerning the well-posedness of the system; the issue is with the uniqueness of Leray-Hopf weak solutions and the existence of global strong solutions. Moreover, the weak solution of the system is not yet proven continuous in time with respect to the initial data. For  NS equations, one way to resolve these issues is  to introduce a nonlinear damping term (\cite{Antontsev+de_2010,Hajduk+Robinson_2017,Tamed_ns, Kinra+Mohan_2024, MTMohan1, Varga}). This has motivated us to modify CHNS system by introducing a  damping term in the third equation of \eqref{equ P}. We  would now like to analyse the influence of this damping term on the well-posedness and strong solutions of the system \eqref{equ P}.

Apart from the introduction of a coupling between the Cahn-Hilliard equation and the damped Navier-Stokes equations, marking the first consideration of such modification in the literature; we provide many novel results for the model \eqref{equ P}. Our work is structured in two parts. In the initial segment, we examine the system described by \eqref{equ P} with regular potential $F$ and non-degenerate mobility $m$. In this scenario, we prove the following results:
\begin{enumerate}
\item[(i)] Existence of a \emph{Leray-Hopf weak solution} for every $r\geq 1$ in three dimensions (3D) (see Theorem \ref{LH weak sol}).
\item[(ii)] Energy equality satisfied by  the weak solutions  for $r\in[3,\infty)$ (cf. Theorem \ref{thm4.10}).
\item[(iii)] We establish the uniqueness of weak solutions for $r>3$ through a regularization of  the initial data $\varphi_0$ (ref. Theorem \ref{unique}).
\item[(iv)] Additionally,  for the critical case $r=3$, the uniqueness and hence the continuous dependence of the initial data is proven when both the viscosity $(\nu)$ and  porosity $(\beta)$ are sufficiently large ($\beta\nu\geq 1$). 
\item[(v)] Subsequently, we employ another technique utilizing the energy equality of weak solutions for an associated system to establish the uniqueness for the case $r\geq3$ without imposing any restriction on $\nu$ and $\beta$ and without considering a regular initial data $\varphi_0$ (Theorem \ref{uni}).
\end{enumerate}
 The inclusion of a damping term provides higher order $\mathrm{L}^p_t\mathbb{L}^q_x$ estimate on weak solutions, consequently establishing the energy equality of weak solution for every $r\geq 3$ without imposing any constraint  in $\nabla \u$ (see \cite{liang2020energy}).  Energy equality immediately implies that the  weak solution is continuous in time as an $\H\times\mathrm{H}^1$-valued function as detailed in Theorem \ref{thm4.10}. When 3D NSE is considered, it is known that if the velocity field satisfies the following Ladyzhenskaya-Prodi-Serrin (LPS) condition:
 \begin{align*}
 	\u\in \mathrm{L}^p(0,T;\mathbb{L}^s_\sigma(\Omega)),\ \frac{2}{p}+\frac{3}{s}\leq 1,\ 3<s\leq\infty,
 \end{align*}
  then the regularity results and uniqueness of Leray-Hopf weak solutions hold for large data and time. We point out here that for the absorption exponent  $r\in[4,\infty)$,  the above LPS condition is satisfied, and the regularity results and uniqueness of Leray-Hopf weak solutions are expected. But we are able to prove the uniqueness of Leray-Hopf weak solutions for  $r\in[3,4)$ and regularity results for $r\in(3,4)$ also. 
 
  In the later part of our study, we consider a logarithmic double-well potential and degenerate mobility, and establish: \begin{enumerate}
\item[(i)] the existence of a weak solution  (Theorem \ref{thm 6.3}) and 
\item[(ii)]  the energy equality  (Theorem \ref{thm 6.4}).
\end{enumerate}
Let us mention some future directions and difficulties to investigate further.
 First we remark some technical difficulties in proving the existence of strong solutions in smooth bounded domains. The main difficulty is that the operators $-\Delta$ and the Helmholtz-Hodge orthogonal projection $\mathrm{P}_{\H}$ do not commute on  bounded domains (\cite{Robinson_2016}) and hence the results obtained in  \cite[Lemma 2.1]{Hajduk+Robinson_2017} may not be useful in this context.   In the more general case of  degenerate mobility and singular potential,  the well-known method is  to approximate the singular potential $F$ and the degenerate mobility $m$ by a family of regular potentials $F_{\epsilon}$ and non-degenerate mobilities $m_\epsilon$.  Then at the level of   approximated problem,  in order  to bound $\nabla\mu_\epsilon(0)$ by $\nabla\mu(0)$, we need to  introduce a suitable cut off function and  build an appropriate Faedo-Galerkin approximation scheme. As a result, we need to pass to limit at three different levels which leads to the derivation of uniform estimates in three parameters and it brings several  subtle and complicated  calculations. Therefore, we have deferred this problem for a future investigation.  Secondly, in order to obtain the uniqueness of weak solutions in the degenerate case,   we  require extra regularity for $\varphi$ (at least $\mathrm{L}^{\infty}_t\mathrm{H}^3_x$) which cannot be derived under current assumptions. Another extension of the current work is to  consider a system where the viscosity is dependent on $\varphi$.


\subsection{Organization of the paper} The paper is organized as follows: We describe the preliminary notations and necessary function spaces needed for this work in the next section. Section \ref{sec3} is devoted for proving the existence of Leray-Hopf weak solutions using Faedo-Galerkin approximations and the energy equality satisfied by weak solutions for the system with regular potential and non-degenerate mobility. Section \ref{sec4} deals with the uniqueness of the solution assuming mobility to be equal to 1 and $r \geq3$. We provide  two different proofs of uniqueness of solutions, one under a restrictive  condition on $\beta$, $\nu$ ($\beta\nu\geq 1$) and initial data $\varphi_0$, which is similar to the proof for the case $ r> 3$ and the second without any extra condition but considering  certain linearized system and exploiting the uniqueness result of that system. Further, in section \ref{sec5},  we prove the existence of a weak solution for a model with degenerate mobility and singular potential; in particular a logarithmic one, by approximating it with a sequence of regular potentials and degenerate mobility by a sequence of non-degenerate mobility and then passing to the limit. The corresponding energy equality is also established.

\section{Mathematical Preliminaries}
Let $\Omega$ be an open bounded subset of $\mathbb{R}^3$ with sufficiently smooth boundary $\partial \Omega$. Here, we  introduce the function spaces that will be useful in the rest of the paper.
\begin{align*}
    \mathbb{C}_c^\infty(\Omega) &:= \big\{\v \in C^\infty(\Omega;\mathbb{R}^3): \text{supp $\v$ is compact}\big\},\\
    \mathbb{D}_\sigma(\Omega) &:= \big\{\v \in \mathbb{C}_c^\infty(\Omega) ; \text{div}~\v = 0\big\},\\
    \mathbb{L}^p_{\sigma}(\Omega) &:= \text{closure of $\mathbb{D}_\sigma(\Omega)$ in $\mathbb{L}^p(\Omega$}),\\
    \V^s_{\mathrm{div}}(\Omega) &:= \text{closure of $\mathbb{D}_\sigma(\Omega)$ in $\mathbb{W}^{s,2}(\Omega)$} \quad \text{for $s > 0$}.
\end{align*}
We denote the Hilbert spaces $\mathbb{L}^2_{\sigma}(\Omega)$ by $\H$ and $ \V^1_{\mathrm{div}}(\Omega)$ by $\V_{\mathrm{div}}.$ We also denote $\mathrm{L}^2$ for $\mathrm{L}^2(\Omega; \mathbb{R})$ The inner product in $\H$ is denoted by $(\cdot,\cdot)$ and the norm is denoted by $\|\cdot\|$. We represent the dual of $\X$ by $\X'$ and the induced duality between $\X$ and $\X'$ by $\langle\cdot,\cdot\rangle$ for any Banach space $\X$.

\subsection{Linear and nonlinear operators}
	Let us define the \emph{Stokes operator} $\A : \D(\A)\cap  \H \to \H$ by 
	\begin{align}
	\label{stokes}
	\A:=-\mathrm{P}_{\H}\Delta,\ \D(\A)=\mathbb{H}^2(\Omega) \cap \V_{\mathrm{div}},\end{align} where $\mathrm{P}_{\H} : \mathbb{L}^2(\Omega) \to \H$ is the \emph{Helmholtz-Hodge orthogonal projection}. We also have
	$$\langle\A\u, \v\rangle = (\u, \v)_{\V_{\mathrm{div}}} = (\nabla\u, \nabla\v) \  \text{ for all } \ \u \in\D(\A), \v \in \V_{\mathrm{div}}.$$
	It should also be noted that  $\A^{-1}: \H \to \H$ is a self-adjoint compact operator on $\H$ and by
	the classical spectral theorem, there exists a sequence $\lambda_j$ with $0<\lambda_1\leq \lambda_2\leq \lambda_j\leq\cdots\to+\infty$ of eigenvalues
	and a family of corresponding eigenfunctions $\boldsymbol{\omega}_j \in \D(\A)$ which is orthonormal in $\H$ such that $\A\boldsymbol{\omega}_j =\lambda_j\boldsymbol{\omega}_j$. We know that $\u$ can be expressed as $\u=\sum\limits_{j=1}^{\infty}(\u,\boldsymbol{\omega}_j) \boldsymbol{\omega}_j,$ so that $\A\u=\sum\limits_{j=1}^{\infty}\lambda_j( \u,\boldsymbol{\omega}_j)\boldsymbol{\omega}_j,$ for all $\u\in\D(\A)$.
 
	For $\u,\v,\w \in \V_{\mathrm{div}},$ we define the trilinear operator $b(\cdot,\cdot,\cdot):\V_{\mathrm{div}}\times \V_{\mathrm{div}}\times \V_{\mathrm{div}}\to\mathbb{R}$ as
	$$b(\u,\v,\w) = \int_\Omega (\u(x) \cdot \nabla)\v(x) \cdot \w(x)\, d x=\sum_{i,j=1}^3\int_{\Omega}u_i(x)\frac{\partial v_j(x)}{\partial x_i}w_j(x)\, d x,$$
	and the bilinear operator $\B:\V_{\mathrm{div}} \times \V_{\mathrm{div}} \to\V_{\mathrm{div}}'$ defined by,
	$$ \langle \B(\u,\v),\w  \rangle = b(\u,\v,\w) \  \text{ for all } \ \u,\v,\w \in \V_\text{{div}}.$$
	Note that an integration by parts yields, 
	\begin{equation}\label{best}
	\left\{
	\begin{aligned}
	b(\u,\v,\v) &= 0, \ \text{ for all } \ \u,\v \in\V_\text{{div}},\\
	b(\u,\v,\w) &=  -b(\u,\w,\v), \ \text{ for all } \ \u,\v,\w\in \V_\text{{div}}.
	\end{aligned}
	\right.\end{equation}
We will denote $\B(\u,\u)$ as $\B(\u)$ for $\mathrm{P}_{\mathbb{H}}[(\u\cdot\nabla)\u]$.

 For every $f \in (\mathrm{H}^1(\Omega))',$ we denote $\overline{f},$ the average of $f$ over $\Omega$, that is, $\overline{f} := |\Omega|^{-1}\langle f, 1 \rangle$, where $|\Omega|$ is the Lebesgue measure of $\Omega$. 
	Let us also introduce the spaces (see \cite{uniqueness})
	\begin{align*}\mathrm{V}_0 &= \{ v \in \mathrm{H}^1(\Omega) \ : \ \overline{v} = 0 \},\\
	\mathrm{V}_0' &= \{ f \in (\mathrm{H}^1(\Omega))' \ : \ \overline{f} = 0 \},\end{align*}
	and the operator $\mathcal{A} : \mathrm{H}^1(\Omega) \rightarrow (\mathrm{H}^1(\Omega))'$  defined by
	\begin{align*}\langle \mathcal{A} u ,v \rangle := \int_\Omega \nabla u(x) \cdot \nabla v(x) \, d x \  \text{ for all } \ u,v \in \mathrm{H}^1(\Omega).\end{align*}
	Clearly $\mathcal{A}$ is linear and it maps $\mathrm{H}^1(\Omega)$ into $\mathrm{V}_0'$ and its restriction $\mathcal{B}$ to $\mathrm{V}_0$ onto $\mathrm{V}_0'$ is an isomorphism.   
	We know that for every $f \in \mathrm{V}_0'$, $\mathcal{B}^{-1}f$ is the unique solution  of the \emph{Neumann problem}:
	$$
	\left\{
	\begin{array}{ll}
	- \Delta u = f, \  \mbox{ in } \ \Omega, \\
	\frac{\partial u}{\partial\boldsymbol{n}} = 0, \ \mbox{ on } \  \partial \Omega.
	\end{array}
	\right.
	$$
	In addition, we have
	\begin{align} \langle \mathcal{A}u , \mathcal{B}^{-1}f \rangle &= \langle f ,u \rangle, \ \text{ for all } \ u\in \mathrm{V},  \ f \in \mathrm{V}_0' , \label{bes}\\
	\langle f , \mathcal{B}^{-1}g \rangle &= \langle g ,\mathcal{B}^{-1}f \rangle = \int_\Omega \nabla(\mathcal{B}^{-1}f)\cdot \nabla(\mathcal{B}^{-1}g)\, d x, \ \text{for all } \ f,g \in \mathrm{V}_0'.\label{bes1}
	\end{align}
	Note that $\mathcal{B}$ can be also viewed as an unbounded linear operator on $\mathrm{H}$ with the
	domain $$\D(\mathcal{B}) = \left\{v \in \mathrm{H}^2(\Omega) : \frac{\partial v}{\partial\boldsymbol{n}}= 0\text{ on }\partial\Omega \right\}.$$
 Owing to \eqref{bes1}, we define $\|f\|_\ast := \|\nabla\mathcal{B}^{-1}f\|$, which defines a norm on $\mathrm{V}_0'$. Moreover, if $u,u' \in \mathrm{L}^2(0, T; \mathrm{V}_0'),$ then we have (\cite[Lemma 1.2, page 176]{Temam}, see also \cite{temam95})
 \begin{align}
     \langle u'(t), \mathcal{B}^{-1}u(t)\rangle = \frac{1}{2}\frac{d}{dt}\|u(t)\|^2_\ast, \quad \text{for a.e. } t\in(0, T).
 \end{align}

 \section{Weak solutions and energy equality: regular potential and non-degenerate mobility }\label{sec3}
 In this section, we study the existence of Leray-Hopf weak solutions of  the system \eqref{equ P} for the case when $F$ is regular and $m$ is non-degenerate. by  using a Faedo-Galerkin approximation technique and compactness arguments. Then we show that every weak solution satisfies the energy equality.  In order to achieve this, we approximate  weak solutions in space and mollify in time and then pass to the limit in both space and time. To this end, we consider the following assumptions on $F$ and $m$:
 \begin{assumption}\label{prop of F} 
The regular potential $F$ and the non-degenerate mobility satisfy the following assumptions:
	\begin{enumerate}
 \item [(1)] $F \in \C^{2}(\mathbb{R})$  and $F \geq 0.$ 
 \item [(2)] There exists $C_0 >0$ such that $F''(s)\geq -C_0$, for all $s \in \mathbb{R}$.
 \item [(3)] There exist $C_1 >0$, $C_2 > 0$ such that $|F'(s)| \leq C_1|s|^p + C_2,$ for all $s \in \mathbb{R}$ and $1 \leq p \leq {5}$.
 \item[(4)] $F(\varphi_0) \in L^1(\Omega)$.
 \item[(5)] There exists $C_3 > 0$, $|F''(s)| \leq C_3(1 + |s|^{p-1})$ for all $s \in \mathbb{R} \text{ where } 1\leq p \leq {5}$.
 \item[(6)]  $m\in C^{0,1}_{\mathrm{loc}}(\mathbb{R})$ and there exist $m_1, \, m_2 >0$ such that $$m_1 \leq m(s) \leq m_2,  \ \text{ for all }\  s \in \mathbb{R}.$$
	\end{enumerate}
\end{assumption}



A typical example of $F$ is given in \eqref{regular} satisfies the above hypothesis on $F$.

\begin{definition}\label{weak sol defn}   
A pair $(\u, \varphi)$ is called  a \emph{Leray-Hopf weak solution} of the system \eqref{equ P} for every $r\geq 1$ on $ \Omega \times (0, T);  0 \leq T<+\infty$  with initial condition $(\u_0, \varphi_0) \in \H \times \mathrm{H}^1$ and forcing $\mathbb{U}\in\mathrm{L}^2(0, T;\V_{\mathrm{div}}')$, if 
   \begin{align*}
       \u &\in  \mathrm{L}^{\infty}(0, T; \H) \cap \mathrm{L}^2(0, T; \V_{\mathrm{div}}) \cap \mathrm{L}^{r+1}(0, T; \mathbb{L}^{r+1}_{\sigma}),\\
       \varphi &\in \mathrm{L}^{\infty}(0, T; \mathrm{H}^1)\cap \mathrm{L}^2(0, T; \mathrm{H}^2),\\
       \u' & \in \mathrm{L}^{\frac{4}{3}}(0, T; \V_{\mathrm{div}}') + \mathrm{L}^{\frac{r+1}{r}}(0, T; \mathbb{L}_\sigma^{\frac{r+1}{r}}),\quad \text{for } 1\leq r < 3,\\
       \u'  & \in \mathrm{L}^2(0, T; \V_{\mathrm{div}}') + \mathrm{L}^{\frac{r+1}{r}}(0, T; \mathbb{L}_\sigma^{\frac{r+1}{r}}), \quad \text{for } r\geq 3\\
       \varphi' & \in \mathrm{L}^2(0, T, (\mathrm{H}^1)'),
   \end{align*}
   and $(\u, \varphi)$ satisfies
   \begin{align}
    -&\int_0^t \langle \boldsymbol{u}(s), \v'(s) \rangle \, ds+ \nu \int_0^t( \nabla\boldsymbol{u}(s), \nabla\v(s) ) \, ds + \int_0^t\langle (\u(s)\cdot\nabla)\u(s), \v(s) \rangle \, ds \no\\
     &\qquad+ \beta \int_0^t\langle |\u(s)|^{r-1}\u(s), \v(s) \rangle \, ds \no\\ &= \int_0^t(\mu(s) \nabla \varphi(s), \v(s)) \, ds + \int_0^t\langle\mathbb{U}(s), \v(s)\rangle ds + ( \boldsymbol{u}_0, \v(0) ) -(\boldsymbol{u}(t), \v(t)) ,\label{weak form u}\\
    -& \int_0^t \langle \varphi(s), \psi'(s)\rangle \, ds + \int_0^t(\boldsymbol{u}(s) \cdot \nabla \varphi(s), \psi(s)) \, ds + \int_0^t( m(\varphi)\nabla \mu(s), \nabla\psi(s)) \, ds \no\\
    &=(\varphi_0, \psi(0))- (\varphi(t), \psi(t)), \label{weak form phi}
   \end{align}
   for all $t \in [0, T]$ and $\v \in \mathbb{D}_{\sigma}(\Omega \times [0, T]), \, \psi \in C_c^\infty({\Omega}\times [0, T])$. Additionally, it holds that 
   \begin{align}\label{in_data_cgts}
       \lim\limits_{t\to 0^+}\|\u(t) -\u_0\|=0, \quad  \text{ and }\ \lim\limits_{t\to 0^+}\|\varphi(t) -\varphi_0\|_{\mathrm{H}^1}=0.
   \end{align}
Moreover,  the following energy inequality satisfied:
    \begin{align}\label{energy inequality1}
         &\frac{1}{2}\bigg(\|\u(t)\|^2 + \|\nabla\varphi(t)\|^2 + 2\int_{\Omega}F(\varphi)dx\bigg) + \nu \int_0^{t} \|\nabla\u(s)\|^2 \, ds + \int_0^{t}\|\u(s)\|_{\mathbb{L}_{\sigma}^{r+1}}^{r+1} ds\nonumber\\&\qquad + \int_0^{t} \|\sqrt{m(\varphi)}\nabla\mu(s)\|^2 ds 
    \no\\& \quad \leq  \frac{1}{2}\bigg(\|\u_0\|^2 + \|\nabla\varphi_0\|^2 + 2\int_{\Omega}F(\varphi_0)dx\bigg)+ \int_0^{t}\langle\mathbb{U}(s), \u(s)\rangle \, ds, 
    \end{align}
    for all $t\in[0,T]$. 
\end{definition}
The following result provides the existence of a Leray-Hopf weak solution of the system \eqref{equ P} for every $r\geq 1$. 
\begin{theorem}[Leray-Hopf weak solution]\label{LH weak sol}
   Let $\Omega$ be a smooth bounded domain in $\mathbb{R}^3$, and the initial condition $(\u_0, \varphi_0) \in \H\times\mathrm{H}^1$,  $\mathbb{U}\in \mathrm{L}^2(0, T; \V_{\mathrm{div}}')$ and $F$ satisfies the Assumption \ref{prop of F}.  Then there exists a Leray-Hopf weak solution of the system \eqref{equ P} satisfying the energy inequality \eqref{energy inequality1} for every $r\geq 1.$
\end{theorem}
Before embarking into the proof of Theorem \ref{LH weak sol}, we provide  some corollaries and a remark   that can be  deduced from  Theorem \ref{LH weak sol}. 
\begin{corollary}\label{cor1}
    If $(\u, \varphi)$ is a Leray-Hopf weak solution of the system  \eqref{equ P}, then one can modify $(\u, \varphi)$ on a set of measure zero in such a way that $(\u, \varphi) \in C([0, T]; \V_{\mathrm{div}}')\times C([0, T];\mathrm{L}^2(\Omega)),$ for $1\leq r\leq 5$ and   $(\u, \varphi) \in C([0, T]; \V_{\mathrm{div}}'+\mathbb{L}^{\frac{r+1}{r}}_{\sigma})\times C([0, T];\mathrm{L}^2(\Omega)),$ for $r>5$. 
\end{corollary}
\begin{proof}
    For $1\leq r <3,$ $1< \frac{4}{3} < \frac{r+1}{r}$,  we have the following continuous embedding: $$\V_{\mathrm{div}}\subset \mathbb{L}^6_{\sigma} \subset \mathbb{L}^4_{\sigma} \subset \mathbb{L}^{r+1}_{\sigma} \subset \mathbb{L}^{\frac{r+1}{r}}_{\sigma} \subset \mathbb{L}^{\frac{4}{3}}_{\sigma} \subset \V_{\mathrm{div}}'.$$ 
    We have from Theorem \ref{LH weak sol},  $\u' \in \mathrm{L}^{\frac{4}{3}}(0, T; \V_{\mathrm{div}}')$. Additionally, we have $\u \in \mathrm{L}^\infty(0, T; \H)$ which implies $\u\in \mathrm{L}^{\frac{4}{3}}(0, T; \V_{\mathrm{div}}')$, so that $\u \in \mathrm{W}^{1, \frac{4}{3}}(0, T; \V_{\mathrm{div}}').$ Therefore from (\cite[Section 5.9, Theorem 2]{Evans}), we deduce that $\u$ is equal  a.e. to an absolutely continuous function from $[0, T]$ to $\V_{\mathrm{div}}'.$
    Similarly if $r\geq 3$, we have $\u' \in \mathrm{L}^{\frac{r+1}{r}}(0, T; \V_{\mathrm{div}}'+\mathbb{L}^{\frac{r+1}{r}}_{\sigma})$ and $\u \in\mathrm{L}^2(0,T;\V_{\mathrm{div}})\subset \mathrm{L}^{\frac{r+1}{r}}(0, T; \V_{\mathrm{div}}'),$ so that $\u\in \mathrm{W}^{1,\frac{r+1}{r}}(0, T; \V'_{\mathrm{div}}+\mathbb{L}^{\frac{r+1}{r}}_{\sigma}),$ and using the same argument as above, we conclude $\u \in C([0, T]; \V'_{\mathrm{div}}+\mathbb{L}^{\frac{r+1}{r}}_{\sigma})$ in this case. Note that for $1\leq r\leq 5$, $\V_{\mathrm{div}}\subset\mathbb{L}^{r+1}_{\sigma}\subset \mathbb{L}^{\frac{r+1}{r}}_{\sigma}\subset \V_{\mathrm{div}}'$, which implies $\u \in C([0, T]; \V'_{\mathrm{div}})$.  For $\varphi,$ one can use the Aubin-Lions lemma to get  $\varphi \in C([0, T]; \mathrm{L}^2(\Omega))$. Indeed,  since $\mathrm{H}^1(\Omega)\subset\mathrm{L}^2(\Omega)\subset(\mathrm{H}^1(\Omega))'$, the embedding $\mathrm{H}^1(\Omega)\subset\mathrm{L}^2(\Omega)$ is compact and the embedding $\mathrm{L}^2(\Omega)\subset(\mathrm{H}^1(\Omega))'$ is continuous,  $\varphi\in\mathrm{L}^{\infty}(0,T;\mathrm{H}^1(\Omega))$ with $\varphi'  \in \mathrm{L}^2(0, T, (\mathrm{H}^1(\Omega))'),$ implies $\varphi\in C([0,T];\mathrm{L}^2(\Omega))$. 
\end{proof}
\begin{corollary}\label{cor4.4}
   Every weak solution $(\u, \varphi)$ of the system \eqref{equ P} are weakly continuous in time $t\in [0, T]$ in $\H\times\mathrm{H}^1$, that is, $(\u, \varphi)\in C_w([0,T];\H)\times C_w([0,T];\mathrm{H}^1).$ 
\end{corollary}
\begin{proof}
 The proof directly follows from (\cite[Proposition 1.7.1] {milani_pascal}) by taking $X = \H\times\mathrm{H}^1$ and $Y = (\V_{\mathrm{div}}'+\mathbb{L}^{\frac{r+1}{r}}_{\sigma})\times (\mathrm{H}^1)'$.    
\end{proof}

\begin{remark}
    From Corollary \ref{cor4.4}, it is clear that the equations \eqref{weak form u} and \eqref{weak form phi} make sense. 
\end{remark}
Let us now provide a proof of Theorem \ref{LH weak sol}. 
\begin{proof}[ Proof of Theorem \ref{LH weak sol}]
We prove the theorem in the following steps:
\vskip 0.1 cm
\noindent\textbf{Step 1:} \emph{Faedo-Galerkin approximation.}
Let $\{\boldsymbol{\omega}_1,   \ldots, \boldsymbol{\omega}_n, \ldots\}$ be a complete orthonormal basis in $\H$ and orthogonal in $ \V_{\mathrm{div}}$,  which are the eigenfunctions of the Stokes operator $\A$ such that $\boldsymbol{\omega}_i\in\D(\A),$ for $i=1,2,\ldots$. 
Also let $\{\psi_1,  \ldots, \psi_n, \ldots \}$ be a complete orthonormal basis in $\mathrm{H}^1$, which are eigenfunctions of the Neumann operator. Let $\H_n=\mathrm{span}\{\boldsymbol{\omega}_1,\ldots,\boldsymbol{\omega}_n\}$ and $\mathrm{V}_n=\mathrm{span}\{\psi_1,\ldots,\psi_n\}$ be the $n$-dimensional subspaces of $\H$ and $\mathrm{H}^1,$ respectively. Let $\P_n$, $\mathrm{P}_n$ be the orthogonal projections from $\V'_{\text{div}}, \mathrm{H}^1$ onto $\H_n, \mathrm{V}_n,$ respectively. We consider  for all $t\in[0,T]$, 
\begin{align*}
    \u_n(\x, t) = \sum_{i=1}^{n}a_n^i(t)\boldsymbol{\omega}_i,  \quad \varphi_n(\x, t) = \sum_{i=1}^{n}b_n^i(t)\psi_i, \quad \mu_n(\x, t) = \sum_{i=1}^{n}c_n^i(t)\psi_i, 
\end{align*}
such that $a_n^i(0)=(\u_0,\boldsymbol{\omega}_i), \, b^i_n(0) = (\varphi_0, \psi_i)$ which satisfies the following system of ODEs:
\begin{equation}\label{projected equ}
\left\{
\begin{aligned}
    & \langle \boldsymbol{u}_n'(t), \v \rangle - \nu (\P_n\Delta\boldsymbol{u}_n(t), \v) + (\P_n(\boldsymbol{u}_n(t)\cdot \nabla \boldsymbol{u}_n(t)), \v ) + \beta ( \P_n(|\u_n|^{r-1}\u_n(t)), \v ) \\&= (\P_n(\mu_n(t) \nabla \varphi_n(t)), \v \rangle + \langle\P_n\mathbb{U}(t), \v\rangle,\\
   &  \langle \varphi_n'(t), \psi\rangle + (\mathrm{P}_n(\boldsymbol{u}_n(t) \cdot \nabla \varphi_n(t)), \psi ) = (\mathrm{P}_n (\text{div}(m(\varphi_n)\nabla \mu_n)), \psi ), \\
   & \mu_n = \mathrm{P}_n(-\Delta\varphi_n + F'(\varphi_n)),\\
    & (\u_n(0), \v) = (\P_n\u(0), \v) , \, (\mathrm{P}_n\mathrm{\varphi}(0), \psi) = (\varphi_n(0), \psi),
\end{aligned}   
\right.
\end{equation}
for all $\v \in \H_n \text{ and } \psi \in \mathrm{V}_n.$ 
Using the local Lipschitz property of the nonlinear terms, one can invoke Carath\'eodory's existence theorem to obtain a   unique local solution to the system \eqref{projected equ}, say $(\u_n, \varphi_n) \in C([0, T'], \H_n)\times C([0, T']; \mathrm{V}_n),$ for some time $0<T'\leq T$. We show that the solution is  global by establishing uniform energy estimates for $(\u_n, \varphi_n)$. 
To achieve this, we utilize  $\u_n$ and $\mu_n$ as our trial functions in $\eqref{projected equ}_1, \eqref{projected equ}_2,$ respectively. So we have the following equations for a.e. $t\in [0, T]$: 
\begin{align}
    &\langle \u'_n(t), \u_n(t) \rangle + \nu(\nabla\u_n(t), \nabla\u_n(t) ) + \langle (\u_n(t)\cdot\nabla)\u_n(t), \u_n(t) \rangle + \beta \langle |\u_n(t)|^{r-1}\u_n(t), \u_n(t) \rangle \no\\
    & \quad = \langle \mu_n(t)\nabla\varphi_n(t), \u_n(t) \rangle + \langle\mathbb{U}_n(t), \u_n(t)\rangle,\label{aprox CBF}\\
   & \langle\varphi'_n(t), \mu_n(t)\rangle + (\u_n(t)\cdot\nabla\varphi_n(t), \mu_n(t)) + \|\sqrt{m(\varphi_n(t))}\nabla\mu_n(t)\|^2 = 0,\label{aprox CH}
\end{align}
where, $\mathbb{U}_n = \P_n\mathbb{U}$.
\vskip 0.1 cm
\noindent\textbf{Step 2:} \emph{Uniform energy estimates:}
From \eqref{aprox CBF} and \eqref{aprox CH}, we can write
\begin{align}
    & \frac{1}{2}\bigg( \|\u_n(t)\|^2 + \|\nabla\varphi_n(t)\|^2 + 2\int_\Omega  F(\varphi_n)dx\bigg) + \nu \int_{0}^{t}\|\nabla\u_n(s)\|^2ds + \beta \int_{0}^{t}\|\u_n(s)\|_{\mathbb{L}_{\sigma}^{r+1}}^{r + 1}ds \no\\&\quad+ \int_{0}^{t}\|\sqrt{m(\varphi_n(s))}\nabla\mu_n(s)\|^2ds \no\\& =  \frac{1}{2}\bigg(\|\u_n(0)\|^2 + \|\nabla\varphi_n(0)\|^2 + 2\int_\Omega  F(\varphi_0) \, dx \bigg) + \int_0^{t}\langle\mathbb{U}_n(t), \u_n(t)\rangle \, ds,\label{aprox energy}
\end{align}
for all $t\in[0,T]$. Note that 
\begin{align*}
    \u^n_0 &:= \P_n\u_0 = \sum_{i=1}^{n}(\u_0, \boldsymbol{\omega}_i )\boldsymbol{\omega}_i, \,
    \varphi^n_0 := \mathrm{P}_n\varphi_0 = \sum_{i=1}^{n}( \varphi_0, \psi_i )\psi_i, \, \mathbb{U}_n = \P_n\mathbb{U}\\
   \text{ so that }\  \|\u^n_0\|^2  &\leq \|\u_0\|^2, \,
    \|\nabla\varphi^n_0\|^2  \leq \|\nabla\varphi_0\|^2 \text{ and } \|\mathbb{U}_n\|_{\V_{\mathrm{div}}'} \leq \|\mathbb{U}\|_{\V_{\mathrm{div}}'}.
\end{align*}
Integrating \eqref{aprox energy} and using  Assumption \ref{prop of F} (1) and (4), we get  for all $t\in[0, T]$:
\begin{align}\label{fin aprox energy}
   & \frac{1}{2}\bigg( \|\u_n(t)\|^2 + \|\nabla\varphi_n(t)\|^2 + 2\int_\Omega  F(\varphi_n)dx\bigg) + \nu \int_{0}^{t}\|\nabla\u_n(s)\|^2ds + \beta \int_{0}^{t}\|\u_n(s)\|_{\mathbb{L}_{\sigma}^{r+1}}^{r + 1}ds \no\\&\qquad+ \int_{0}^{t}\|\sqrt{m(\varphi_n(s))}\nabla\mu_n(s)\|^2ds   \no\\&\quad\leq \frac{1}{2}\Big(\|\u_0\|^2 + \|\nabla\varphi_0\|^2 + 2\int_\Omega  F(\varphi_0) \, dx \Big) + \|\mathbb{U}\|_{\mathrm{L}^2(0, T; \V_{\mathrm{div}}')}, 
\end{align}
and the right-hand side is independent of $n$. 
\vskip 0.1 cm
\noindent\textbf{Step 3:} \emph{Extraction of subsequences:}
We observe that 
\begin{align}
    \| |\u_n|^{r-1}\u_n \|^{\frac{r+1}{r}}_{\mathbb{L}_{\sigma}^{\frac{r+1}{r}}} = \|\u_n\|^{r+1}_{\mathbb{L}_{\sigma}^{r+1}}.\label{absorption}
\end{align}
Since $|\langle\varphi_n(t)\rangle|=|\langle\varphi_n(0)\rangle|\leq\|\varphi_n(0)\|_{\mathrm{L}^1}\leq |\Omega|^{1/2}\|\varphi_0\|_{\mathrm{L}^2}<\infty$ for all $t\in(0,T)$, an application of the Poincar\'e-Wirtinger inequality  implies $\|\varphi_n(t)\|_{\mathrm{L}^2}\leq C,$ for all $t\in[0,T]$, where $C$ is independent of $n$. By leveraging expression \eqref{fin aprox energy}, \eqref{absorption} along with the Banach-Alaoglu theorem, we can extract subsequences $\{\u_{n_k}\}_{k\in\N}, \, \{\varphi_{n_k}\}_{k\in\N}, \text{ and } \{\mu_{n_k}\}_{k\in\N}$ (we will continue to use the notation as  $\{\u_n\}, \, \{\varphi_n\}, \text{ and } \{\mu_n\}$) such that following convergences hold:
\begin{equation}\label{convergence1}
\left\{
    \begin{aligned}
      \u_n & \xrightharpoonup{*} \u \text{ in } \mathrm{L}^\infty(0, T; \H),  \\
   \u_n & \rightharpoonup \u \text{ in } \mathrm{L}^2(0, T; \V_{\mathrm{div}}) , \\
    \u_n & \rightharpoonup \u \text{ in } \mathrm{L}^{r+1}(0, T; \mathbb{L}_\sigma^{r +1}), \\
    |\u_n|^{r-1}\u_n & \rightharpoonup \z \text{ in } \mathrm{L}^{\frac{r+1}{r}}(0, T; \mathbb{L}_\sigma^{\frac{r +1}{r}}) ,\\
    \varphi_n & \xrightharpoonup{*} \varphi \text{ in } \mathrm{L}^\infty(0, T; \mathrm{H}^1) ,  \\
   F(\varphi_n) & \xrightharpoonup{*} F^\ast \text{ in } \mathrm{L}^\infty(0, T; \mathrm{L}^1) ,\\
    \mu_n & \rightharpoonup \mu \text{ in } \mathrm{L}^2(0, T; \mathrm{H^1}) . 
    \end{aligned}
\right.
\end{equation}
Also, by an integration by parts and using Assumption \ref{prop of F} (2), we observe 
\begin{align*}
   \frac{1}{2} \|\nabla\mu_n \|^2 + \frac{1}{2}\|\nabla\varphi_n \|^2 &\geq (\mu_n, -\Delta \varphi_n) \geq \|\Delta\varphi_n\|^2 - C_0 \|\nabla\varphi_n\|^2, 
\end{align*}
which implies
\begin{align}\label{aprox del phi}
    \frac{1}{2} \|\nabla\mu_n \|^2 & \geq \|\Delta\varphi_n\|^2 - \bigg(C_0 +\frac{1}{2}\bigg) \|\nabla\varphi_n\|^2.
\end{align}
Inequality \eqref{aprox del phi} together with \eqref{aprox energy}, the fact that $m_1\leq m(\varphi_n)\leq m_2$ and the elliptic regularity results (cf. \cite{Brezis}) provide the convergence 
\begin{align}
    \varphi_n & \rightharpoonup \varphi \text{ in } \mathrm{L}^2(0, T; \mathrm{H}^2) . \label{phi weak}
\end{align}
\vskip 0.1 cm
\noindent\textbf{Step 4:} \emph{Time derivative estimates:}
To estimate the time derivative, $\u'_n$, we write 
\begin{align}
    \u'_n &= \u^1_n + \u^2_n,\  \text{ where} \\
     \u^1_n &:=  \nu \P_n\Delta\boldsymbol{u}_n - \P_n((\u_n\cdot\nabla)\u_n) + \P_n(\mu_n\nabla\varphi_n) + \mathbb{U}_n, \ \text{ and }\  \u^2_n := \beta\P_n(|\u_n|^{r-1}\u_n).\label{part u}
\end{align}
Our aim is to show $\u^1_n \in \mathrm{L}^{\frac{4}{3}}(0, T; \V_{\mathrm{div}}')$ and $\u^2_n \in \mathrm{L}^{\frac{r+1}{r}}(0, T; \mathbb{L}_\sigma^{\frac{r+1}{r}})$. For $1\leq r<3$, we take $\v \in \V_{\mathrm{div}}$ with $\|\v\|_{\V_{\mathrm{div}}} \leq 1$ and write $\v= \v^1 + \v^2$, where $\v^1 \in \mathrm{span}\{\boldsymbol{\omega}_k\}_{k=1}^n$ and $(\v^2, \boldsymbol{\omega}_k) =0; \, k=1, \ldots, n$. Since the functions $\{\boldsymbol{\omega}_k\}_{k=1}^\infty$ are orthogonal in $\V_{\mathrm{div}}$, $\|\v^1\|_{\V_{\mathrm{div}}} \leq \|\v\|_{\V_{\mathrm{div}}} \leq 1$. Utilizing \eqref{aprox CBF} and expression of $\u_n^1,$ we derive 
\begin{align*}
    | \langle \u^1_n, \v \rangle| & \leq \nu \|\u_n\|_{\V_{\mathrm{div}}}\|\v\|_{\V_{\mathrm{div}}} + C\|\u_n\|_{\H}^\frac{1}{2}\|\u_n\|^\frac{3}{2}_{\V_{\mathrm{div}}}\|\v\|_{\V_{\mathrm{div}}}+C\|\nabla\mu_n\| \|\varphi_n\|^{\frac{1}{2}}\|\varphi_n\|^{\frac{1}{2}}_{\mathrm{H}^1}\|\v\|_{\V_{\mathrm{div}}}\nonumber\\&\quad+\|\mathbb{U}_n\|_{\V_{\mathrm{div}}'}\|\v\|_{\V_{\mathrm{div}}} \no\\
    & \leq \nu \|\u_n\|_{\V_{\mathrm{div}}} + C\|\u_n\|_{\H}^\frac{1}{2}\|\u_n\|^\frac{3}{2}_{\V_{\mathrm{div}}}  +C\|\nabla\mu_n\| \|\varphi_n\|^{\frac{1}{2}}\|\varphi_n\|^{\frac{1}{2}}_{\mathrm{H}^1}+\|\mathbb{U}_n\|_{\V_{\mathrm{div}}'},
    \end{align*}
    since $\|\v\|_{\V_{\mathrm{div}}} \leq 1,$ where we have used H\"older's and Gagliardo-Nirenberg's inequalities. 
    Consequently, for all $\v\in\mathrm{L}^4(0,T;\V_{\mathrm{div}})$, we have 
    \begin{align}\label{u_n^1 norm}
 &\left|\int_0^T \langle \u^1_n(t), \v(t)\rangle dt\right| \nonumber\\ 
   & \leq \bigg[\nu\bigg(\int_0^T\|\u_n(t)\|^2_{\V_{\mathrm{div}}}\, dt\bigg)^{\frac{1}{2}} 
 + \Big( \int_0^T \|\nabla\mu_n(t)\|^2 \|\varphi_n(t)\|^2_{\mathrm{H}^1} \, dt\Big)^{\frac{1}{2}} + \Big(\int_0^T\|\mathbb{U}_n(t)\|^2_{\V_{\mathrm{div}}'} \, dt \Big)^\frac{1}{2}\bigg]\nonumber\\&\qquad\times\bigg(\int_0^T\|\v(t)\|^2_{\V_{\mathrm{div}}} \, dt \bigg)^{\frac{1}{2}} + \bigg(\int_0^T\|\u_n(t)\|^{\frac{3}{2}}_{\H}\|\u_n(t)\|^2_{\V_{\mathrm{div}}} \, dt \bigg)^\frac{1}{2}\bigg(\int_0^T\|\v(t)\|^4_{\V_{\mathrm{div}}} \, dt \bigg)^{\frac{1}{4}}.
\end{align}
For $1\leq r < 3$,  
we know that the embedding $\V_{\mathrm{div}} \hookrightarrow \mathbb{L}^{r+1}_\sigma$ is dense. Decomposing $\v$ in a similar fashion and proceeding as above, we have
\begin{align*}
    |\langle \u^2_n, \v \rangle|  = |\beta\langle|\u_n|^{r-1}\u_n, \P_n\v\rangle| \,  
    \leq  \beta\| |\u_n|^{r-1}\u_n\|_{\mathbb{L}_{\sigma}^{\frac{r+1}{r}}}\|\P_n\v\|_{\mathbb{L}_{\sigma}^{r+1}} .
   \end{align*}
   Therefore, for all $\v\in \mathrm{L}^{r+1}(0, T; \mathbb{L}_{\sigma}^{r+1})$, we have 
   \begin{align}
   	\left|\int_0^T\langle \u^2_n(t), \v(t) \rangle dt \right|
   & \leq \beta\bigg( \int_0^T \|\u_n(t)\|^{r+1}_{\mathbb{L}^{r+1}_{\sigma}} \bigg)^{\frac{r}{r+1}}\left(\int_0^T\|\v(t)\|_{\mathbb{L}_{\sigma}^{r+1}}^{r+1}\right).\label{u_n^2 norm}
\end{align}
From \eqref{u_n^1 norm} and \eqref{u_n^2 norm}, we conclude  for $1\leq r < 3$,
\begin{align}\label{316}
    \{\u'_n\} \text{ is uniformly bounded in } \mathrm{L}^\frac{4}{3}(0, T; \V_{\mathrm{div}}') + \mathrm{L}^{\frac{r+1}{r}}(0, T; \mathbb{L}_{\sigma}^{\frac{r+1}{r}})\subset \mathrm{L}^{\frac{4}{3}}(0, T; \V_{\mathrm{div}}'). 
\end{align}
For $r \geq 3$, again we write $\u'_n= \u^1_n + \u^2_n$  as defined in \eqref{part u} and calculate the followings: 
\begin{align}
& \left|\int_0^T \langle \u^1_n(t), \v(t)\rangle dt\right| \leq  \nu \int_0^T\|\nabla\u_n(t)\|\|\nabla\v(t)\| \, dt+ \|\u_n(t)\|_{\mathbb{L}_{\sigma}^{r+1}}\|\u_n(t)\|_{\mathbb{L}_{\sigma}^{\frac{2(r+1)}{r-1}}}\|\v(t)\|_{\V_{\mathrm{div}}}\no\\&\quad+\int_0^T\|\nabla\mu_n(t)\|\|\varphi(t)\|_{\mathrm{L}^4}\|\v(t)\|_{\mathbb{L}_{\sigma}^4} \, dt + \int_0^T\|\mathbb{U}_n(t)\|_{\V_{\mathrm{div}}'}\|\v(t)\|_{\V_{\mathrm{div}}} \, dt\no\\  
  &  \leq \Big[\nu\Big(\int_0^T\|\u_n(t)\|^2_{\V_{\mathrm{div}}} \, dt\Big)^{\frac{1}{2}}+ \Big(\int_0^T\|\u_n(t)\|^{\frac{2(r+1)}{r-1}}_{\mathbb{L}_{\sigma}^{r+1}}\|\u_n(t)\|^{\frac{2(r-3)}{r-1}}_{\H} \, dt \Big)^\frac{1}{2} +\Big(\int_0^T \|\mathbb{U}_n(t)\|_{\V_{\mathrm{div}}'}^2 \, dt \Big)^{\frac{1}{2}} \no\\&\quad + \Big(\int_0^T\|\nabla\mu_n(t)\|^2\|\varphi_n(t)\|^2_{\mathrm{H}^1} \, dt \Big)^{\frac{1}{2}} \Big]\Big(\int_0^T\|\v(t)\|^2_{\V_{\mathrm{div}}} \, dt\Big)^\frac{1}{2}, \quad \forall \v\in \mathrm{L}^2(0, T;\V_{\mathrm{div}}),
    \end{align}
and 
\begin{align}
\left|\int_0^T\langle \u^2_n(t), \v(t) \rangle dt \right|\leq  
\beta\bigg( \int_0^T \|\u_n(t)\|^{r+1}_{\mathbb{L}^{r+1}_{\sigma}} \bigg)^{\frac{r}{r+1}}\left(\int_0^T\|\v(t)\|_{\mathbb{L}_{\sigma}^{r+1}}^{r+1}\right)^{\frac{1}{r+1}}, \quad \forall \v\in\mathrm{L}^{r+1}(0,T;\mathbb{L}^{r+1}_{\sigma}).
\end{align}
This gives for $r\geq 3$,
\begin{align}\label{time cgts u}
    \{\u'_n\} \text{ is uniformly bounded in } \mathrm{L}^2(0, T; \V_{\mathrm{div}}') + \mathrm{L}^{\frac{r+1}{r}}(0, T; \mathbb{L}_{\sigma}^{\frac{r+1}{r}})\subset \mathrm{L}^{\frac{r+1}{r}}(0, T;\V_{\mathrm{div}}'+ \mathbb{L}_{\sigma}^{\frac{r+1}{r}}).
\end{align}
To estimate $\varphi'_n$,  we first recall the expression from $\eqref{projected equ}_3$ and bound the following using integration by parts and Sobolev embeddings:
\begin{align}
&\left|\int_0^T\langle\varphi'_n(t),\psi(t)\rangle dt\right| \nonumber\\&\leq \int_0^T\|\u_n(t)\|_{\mathbb{L}^3_{\sigma}}\|\varphi_n(t)\|_{\mathrm{L}^3}\|\nabla\psi(t)\| \, dt + \int_0^T\|\sqrt{m(\varphi_n(t))}\nabla\mu_n(t)\|\|\nabla\psi(t)\| \, dt\no\\
& \leq\Big[\bigg(\int_0^T\|\u_n(t)\|^2_{\V_{\mathrm{div}}} \|\varphi_n(t)\|\|\nabla\varphi_n(t)\| \, dt\bigg)^\frac{1}{2} + \bigg(\int_0^T\|\sqrt{m(\varphi_n(t))}\nabla\mu_n(t)\|^2 \, dt\bigg)^\frac{1}{2}\Big]\no\\
&\quad \times \bigg(\int_0^T\|\nabla\psi(t)\|^2 \, dt\bigg)^\frac{1}{2},  \quad \forall \psi\in\mathrm{L}^2(0,T;\mathrm{H}^1).
    \end{align}
This immediately concludes 
\begin{align}\label{time cgts}
      \varphi'_n & \rightharpoonup \varphi' \ \text{ in }\  \mathrm{L}^2(0, T; (\mathrm{H}^1)').
    \end{align}

\vskip 0.1 cm
\noindent\textbf{Step 5:} \emph{Strong convergences:}
With the help of \eqref{convergence1}, \eqref{time cgts}, \eqref{316} and \eqref{time cgts u}, and the Aubin-Lions compactness lemma (\cite[Theorem 5]{JSimon}), we have  the following strong convergences along a subsequence (using the same notation):
\begin{align}
    \u_n & \rightarrow \u \ \text{ in }\  \mathrm{L}^2(0, T; \H),\label{u strong}\\
    \varphi_n & \rightarrow \varphi \ \text{ in }\  \mathrm{L}^2(0, T; \mathrm{H^1}) \ \text{ and }\ \varphi_n  \rightarrow \varphi \ \text{ in }\ C([0,T];\mathrm{L}^2).\label{phi strong}
\end{align}
Using the strong convergence of $\u_n$ in $\eqref{u strong},$ we can also choose a further subsequence such that
\begin{align}
    \u_n(t) \rightarrow \u(t) \text{ in } \H \, \text{ for a.e. } t \in (0, T).  
\end{align}
In fact, along a further subsequence (denoted using the same notation), we have 
\begin{align}\label{424}\mbox{$\u_n\to \u$ \ for a.e.\ $(x,t)\in\Omega\times(0,T)$\  and \ $\varphi_n\to\varphi$\  for a.e. \ $x\in\Omega$ and all $t\in[0,T].$ }
\end{align}

\vskip 0.1 cm
\noindent\textbf{Step 6:} \emph{Passage to limit:}
Owing to the convergence \eqref{424} and (\cite[Lemma 1.3]{Lions}) we obtain
\begin{align}\label{absorption cgts}
    |\u_n|^{r-1}\u_n & \rightharpoonup |\u|^{r-1}\u \ \text{ in }  \ \mathrm{L}^{\frac{r+1}{r}}(0, T; \mathbb{L}_\sigma^{\frac{r +1}{r}}).
\end{align}
 Then from uniqueness of weak limit we conclude $\z = |\u|^{r-1}\u$ (see \eqref{convergence1}).  Moreover, using pointwise convergence of $\varphi$ in \eqref{424} and continuity of $F$, from the convergence in $\eqref{convergence1}_6$ and the uniqueness of weak limit, we get $F^\ast = F(\varphi).$  Now we can pass to the limit as $n\to \infty$ in the approximated system \eqref{projected equ}. For this purpose, we multiply the equation $\eqref{projected equ}_1 \text{ and } \eqref{projected equ}_2$ by $\chi \in C_c^\infty([0, T))$ with $\chi(T) = 0$ and then integrate these equation from $0$ to $T$. Passing to the limit in the linear terms using the convergences given by \eqref{convergence1} is standard. To pass to the limit in the convective term in the CBF equation (damped Navier-Stokes), one can use the convergence \eqref{absorption cgts} and arguments similar to the one in \cite{robinson,Hajduk+Robinson_2017,MTMohan,MTMohan1}. 
 The pointwise convergence in \eqref{424}, continuity of $m$ and LDCT ensure that
\begin{align}\label{mobility cgts}
    m(\varphi_n) \rightarrow m(\varphi) \text{ strongly in } \mathrm{L}^p( \Omega\times(0, T)) \ \  \text{ for all }\ p \geq 1. 
\end{align}
Thus we have 
\begin{align}
    \int_0^T(m(\varphi_n(t))\nabla\mu_n(t), \chi(t)\nabla\psi)dt \rightarrow \int_0^T(m(\varphi(t))\nabla\mu(t), \chi(t)\nabla\psi)dt \quad \text{as } n \rightarrow \infty.
\end{align}
Using the strong convergences of $\u_n, \varphi_n$ passing to the limit in the other nonlinear terms like those in CHNS   system is standard  and details can be found in \cite{colli_weak}. Therefore, we can pass to limit in \eqref{projected equ} and the limit satisfies \eqref{weak form u} and \eqref{weak form phi}. 
\vskip 0.1 cm
\noindent\textbf{Step 7:} \emph{Energy inequality:}
To show the energy inequality satisfied by the limit function $(\u, \varphi)$, we need to pass to the limit in \eqref{fin aprox energy}.
Using pointwise convergence of $\u_n, \varphi_n$ in \eqref{424} and weak convergence of $\u_n$ in $\eqref{convergence1}_2, \, \eqref{convergence1}_3$, and the weak lowersemicontinuity property of the norms, we obtain 
\begin{align}\label{liminf1}
    &\liminf_{n\rightarrow\infty}\bigg\{\frac{1}{2}\big( \|\u_n(t)\|^2 + \|\nabla\varphi_n(t)\|^2 \big)+ \int_\Omega F(\varphi_n)dx + \nu \int_{0}^{t}\|\nabla\u_n(s)\|^2ds + \beta \int_{0}^{t}\|\u_n(s)\|^{r + 1}_{\mathbb{L}^{r+1}_{\sigma}}ds\bigg\} \no\\
    &\geq \frac{1}{2}\big( \|\u(t)\|^2 + \|\nabla\varphi(t)\|^2\big) + \int_\Omega F(\varphi)dx + \nu \int_{0}^{t}\|\nabla\u(s)\|^2ds + \beta \int_{0}^{t}\|\u(s)\|^{r + 1}_{\mathbb{L}^{r+1}_{\sigma}}ds.
\end{align}
We also have 
\begin{align}
    \sqrt{m(\varphi_n)}\nabla\mu_n \rightarrow \sqrt{m(\varphi)}\nabla\mu \ \text{ weakly } \ \mathrm{L}^2(\Omega\times(0, T)),
\end{align}
which is a consequence of \eqref{mobility cgts} and $\eqref{convergence}_7$. Thus once again using the weakly lowersemicontinuity property of the norms,  we deduce
\begin{align}\label{liminf2}
    \liminf_{n\rightarrow\infty}\int_0^t\|\sqrt {m(\varphi_n(s))}\nabla\mu_n(s)\|^2 ds \geq \int_0^t\|\sqrt{ m(\varphi(s))}\nabla\mu(s)\|^2 ds.
\end{align}
Using the strong convergence in \eqref{u strong} and the fact that $\U_n\to\U$ in $\mathrm{L}^2(0, T;\V'_{\text{div}})$, we have
\begin{align}\label{U cgts}
    \lim_{n \to \infty }\int_0^t\langle\mathbb{U}_n(s), \u_n(s)\rangle ds \to \int_0^t\langle\mathbb{U}(s), \u(s)\rangle ds,
\end{align}
Combining \eqref{liminf1}, \eqref{liminf2} and \eqref{U cgts},
one can derive the energy inequality \eqref{energy inequality1}.

\vskip 0.1 cm
\noindent\textbf{Step 8:} \emph{Convergence of the initial data:}
It follows from  Corollary \ref{cor4.4} that
\begin{align}\label{wcs}
  (\u(t), \v)\to (\u(0), \v), \quad  \langle\varphi(t), \psi\rangle \to \langle\varphi(0), \psi), \text{ as } t\to 0^+,
\end{align}
 for all $\v\in \H$ and $\psi\in (\mathrm{H}^1)'.$ Also from the energy inequality \eqref{energy inequality1} we obtain
\begin{align*}
\limsup_{t\to 0^+}\frac{1}{2}\bigg(\|\u(t)\|^2 + \|\nabla\varphi(t)\|^2 + 2\int_{\Omega}F(\varphi)dx\bigg)\leq  \frac{1}{2}\bigg(\|\u_0\|^2 + \|\nabla\varphi_0\|^2 + 2\int_{\Omega}F(\varphi_0)dx\bigg).
\end{align*}
  On the other hand, using the fact that $\u(t)$ is weakly continuous in $\H$ and $\varphi(t)$ is weakly continuous in $\mathrm{H}^1$ (Corollary \ref{cor4.4}), one can deduce 
\begin{align*}
\liminf_{t\to 0^+}\frac{1}{2}\bigg(\|\u(t)\|^2 + \|\nabla\varphi(t)\|^2\bigg) \geq \frac{1}{2}\bigg(\|\u_0\|^2 + \|\nabla\varphi_0\|^2\bigg).
\end{align*}
Furthermore, using the strong convergence of $\varphi$,  \eqref{phi strong} and Assumption \ref{prop of F}  (1),  along with LDCT yield
\begin{align*}
\lim_{t\to 0^+}\int_\Omega F(\varphi(t)) \, dx = \int_\Omega F(\varphi_0) \, dx.
\end{align*}
Combining all the above estimates we obtain
\begin{align*}
\lim_{t\to 0^+}\frac{1}{2}\bigg(\|\u(t)\|^2 + \|\nabla\varphi(t)\|^2 + 2\int_{\Omega}F(\varphi)dx\bigg) = \frac{1}{2}\bigg(\|\u_0\|^2 + \|\nabla\varphi_0\|^2 + 2\int_{\Omega}F(\varphi_0)dx\bigg),
\end{align*}  which again imply 
\begin{align}\label{ncs}
   \|\u(t)\| \to \|\u_0\|, \quad \|\nabla\varphi(t)\| \to \|\nabla\varphi_0\| \ \text{ as }\ t\to0^+.
\end{align}
Since $\varphi\in C([0,T];\L^2),$ it is immediate that $\|\varphi(t)-\varphi_0\|$ as $\to 0^+$.  Combining \eqref{wcs} and \eqref{ncs}, we have \eqref{in_data_cgts}, as weak convergence  together with  the strong convergence in norm in Hilbert spaces imply the strong convergence.
\end{proof}

\begin{remark}\label{rem4.7}
  Since $m_1\leq m(\varphi)\leq m_2,$  from \eqref{fin aprox energy}, it follows that  $\nabla\mu$ is bounded in $\mathrm{L}^2(0, T;\mathbb{L}^2)$. Owing to Assumption \ref{prop of F} (3), we obtain from the expression  of $\mu=-\Delta\varphi+F'(\varphi)$ and Sobolev's inequality that \begin{align*}|\overline{\mu}|&=|\Omega||(\mu, 1)|\leq |\Omega|\int_{\Omega}|F'(\varphi)| dx \leq |\Omega|C_1\int_{\Omega}|\varphi|^pdx+|\Omega|^2C_2\nonumber\\&\leq  C\|\varphi\|_{\mathrm{H}^1}^p+C<\infty, \text{ for a.e. }t\in (0, T), \end{align*} for all $1\leq p\leq 6$. Then by using the Poincar\'e-Wirtinger inequality inequality, we infer that $\mu\in\mathrm{L}^2(0,T;\mathrm{H}^1)$.
\end{remark}

Now, we  perform some a-priori and a-posterior estimate of $\varphi$ to show the higher regularity of $\varphi$, which we need in the uniqueness result. We introduce the following necessary lemmas:
\begin{lemma}\label{higher regularity phi}
Let $(\u, \varphi)$ be a Leray-Hopf weak solution of the system \eqref{equ P}. For any initial data $(\u_0, \varphi_0) \in \H \times \mathrm{H}^1,$ the relative concentration of fluids, $\varphi \in \mathrm{L}^4(0, T; \mathrm{H}^2).$
\end{lemma}

\begin{proof}
    Consider the following  equation for a.e. $t\in(0,T)$ in $\mathrm{L}^2$:
    \begin{align*}
        \mu &= -\Delta\varphi + F'(\varphi).
    \end{align*}
  Multiplying both side of the equation by $\Delta\varphi$ gives rise to the following equation:
  \begin{align*}
      \|\Delta\varphi\|^2 & = (F'(\varphi), \Delta\varphi) - (\mu, \Delta\varphi) \\
      & \leq \|\nabla\mu\| \|\nabla\varphi\| + C_0\|\nabla\varphi\|^2,
  \end{align*}
 where we have used the Assumption \ref{prop of F} (2).  Therefore, taking the square on both sides of the above equation and then integrating it over the time interval $(0, T)$,
  \begin{align*}
      \int_0^T\|\Delta\varphi(t)\|^4 dt &\leq \int_0^T\|\nabla\mu(t)\|^2 \|\nabla\varphi(t)\|^2 dt + C_0\int_0^T\|\nabla\varphi(t)\|^2 dt\\
      & \leq \|\nabla\mu\|^2_{\mathrm{L}^2(0, T; \mathrm{L}^2)}\|\varphi\|^2_{\mathrm{L}^\infty(0, T; \mathrm{H}^1)} + C_0 \|\varphi\|^4_{\mathrm{L}^\infty(0, T; \mathrm{H}^1)}.
  \end{align*}
 By the elliptic regularity theory (cf. \cite[Theorem 9.26]{Brezis}), one can conclude that $\varphi \in \mathrm{L}^4(0, T; \mathrm{H}^2).$ 
\end{proof}

\begin{remark}
		Note that in Lemma \ref{lem4.9} below, we have taken a stronger assumption on $\varphi_0$. Therefore, in that context we need $ 1\leq p\leq 4$ in Assumption \ref{prop of F} (3) and (5).
\end{remark}

\begin{lemma}\label{lem4.9}
    Let $(\u, \varphi)$ be a Leray-Hopf weak solution of \eqref{equ P} with $m\equiv 1$ and the initial condition $(\u_0, \varphi_0) \in \H \times \mathrm{H}^2$ so that $\mu_0 = -\Delta\varphi_0 + F'(\varphi_0)\in \mathrm{L}^2$. Then the chemical potential $\mu$ satisfies $\mu \in \mathrm{L}^\infty(0, T; \mathrm{L}^2) \cap \mathrm{L}^2(0, T; \mathrm{H}^2).$
\end{lemma}
\begin{proof}
   This can be shown by following \cite[Lemma 5.5]{bosia_pullback_attractor}. For the sake of completeness, we provide a proof. We consider the equation $\eqref{projected equ}_3$. Then we find the evolution equation satisfied by $\mu_n$, that is, we formally differentiate equation $\eqref{projected equ}_3$  and take the inner product with $\mu_n$ in $\mathrm{L}^2$. For simplicity of notation, we skip the suffix $n$, but all the calculations have to be understood at the Faedo-Galerkin approximation level. Therefore,  we get 
    \begin{align*}
        \langle\mu', \mu\rangle + ( \Delta\mu, \Delta\mu) &= (\u\cdot\nabla\varphi, \Delta\mu) - \langle F''(\varphi)\u\cdot\nabla\varphi, \mu\rangle + \langle F''(\varphi)\Delta\mu, \mu\rangle \nonumber\\&=(\u\cdot\nabla\varphi, \Delta\mu) - \langle \nabla(F'(\varphi))\cdot\u, \mu\rangle + \langle F''(\varphi)\Delta\mu, \mu\rangle. 
    \end{align*}
   where the boundary condition $\eqref{equ P}_5$ imply that $\frac{\partial\varphi_t}{\partial\n} = 0.$ So by using integration by parts, repeatedly applying Agmon's inequality, Gagliardo-Nirenberg inequality, and Assumption \ref{prop of F} on $F$, we have the following inequality:
  \begin{align*}
        &\frac{1}{2}\frac{d}{dt}\|\mu\|^2 + \|\Delta\mu\|^2 \no\\& \leq \|\u\|_{\mathbb{L}^4_\sigma}\|\nabla\varphi\|_{\mathbb{L}^4}\|\Delta\mu\| + \|F'(\varphi)\|_{\mathrm{L}^\infty}\|\u\|\|\nabla\mu\| +\| F''(\varphi)\|_{\mathrm{L}^\infty}\|\Delta\mu\| \|\mu\|\no\\
        &\leq C\|\u\|^\frac{1}{2}\|\nabla\u\|^\frac{1}{2}\|\varphi\|_{\mathrm{H}^1}^{\frac{1}{4}}\|\Delta\varphi\|^{\frac{3}{4}}\|\Delta\mu\| + C\|\varphi\|^{p}_{\mathrm{L}^\infty}\|\u\|\|\nabla\mu\| + C\|\varphi\|^{p-1}_{\mathrm{L}^\infty}\|\Delta\mu\|\|\mu\|\no\\
        & \leq C \|\u\|^2\|\nabla\u\|^2 + C\|\varphi\|_{\mathrm{H}^1}\|\Delta\varphi\|^3 + \epsilon\|\Delta\mu\|^2 + C \|\varphi\|^{p}_{\mathrm{H}^1}\|\varphi\|^{p}_{\mathrm{H}^2}\|\u\|^2+ \|\nabla\mu\|^2 + \|\nabla\u\|^2 \no\\&\quad+C \|\varphi\|^{p-1}_{\mathrm{H}^1}\|\varphi\|^{p-1}_{\mathrm{H}^2}\|\mu\|^2 + \delta\|\Delta\mu\|^2.
    \end{align*}
    For $\epsilon+\delta> 0$ small enough, integrating the resulting inequality over the interval $[0, T]$ and applying Gronwall's inequality yield the required result.
\end{proof}
\begin{remark}
     From Lemma \ref{lem4.9}, using interpolation inequality, we can further conclude that $\mu \in \mathrm{L}^4(0, T; \mathrm{H}^1)$
\end{remark}

\subsection{Energy Equality}
Let us now discuss the energy equality satisfied by a weak solution $(\u, \varphi).$ It should be noted that, given the regularity of $\varphi $ and $\mu$ in Definition \ref{weak sol defn} and Remark \ref{rem4.7}, one can replace $\psi \text{ by } \mu$ in the equation \eqref{weak form phi}. But, replacing $\v \text{ by } \u$ in \eqref{weak form u} is not so straightforward. For this purpose we approximate $\u(\cdot)$ by the following the approximation given in \cite{robinson} (see \cite{Hajduk+Robinson_2017} for periodic domains). We define
\begin{align}\label{eigensp approx}
        \u^n(t) = \sum_{\lambda_j < n^2}e^{-\frac{\lambda_j}{n}}\langle \u(t), \boldsymbol{\omega}_j\rangle\boldsymbol{\omega}_j, \quad \text{for each } t \in [0, T],
\end{align}
where $\{\boldsymbol{\omega}_j\}_{j=1}^n$ are first $n$ eigenfunctions of the Stokes operator.  In \cite{robinson}, the authors studied that the approximation given by \eqref{eigensp approx} satisfies the following properties:
\begin{enumerate}
    \item $\u^n(t) \rightarrow \u(t)$ in $\H_0^1(\Omega)$  with $\|\u^n(t)\|_{\H^1} \leq C \|\u(t)\|_{\H^1}$ for all $t \in [0, T]$.
    \item  $\u^n(t) \rightarrow \u(t)$ in $\mathbb{L}^p_\sigma(\Omega)$  with $\|\u^n(t)\|_{\mathbb{L}^p_\sigma(\Omega)} \leq C \|\u(t)\|_{\mathbb{L}^p_\sigma(\Omega)}$ for any $1 < p < \infty$ and all $t \in [0, T]$.
    \item $\text{ div }\u^n(t) = 0$ in $\Omega$ and $\u^n(t) = 0$ on $\partial\Omega$ for all $t \in [0, T]$.
\end{enumerate}
Since $\boldsymbol{\omega}_j$ are eigenfunctions of Stokes operator $\A$, $\boldsymbol{\omega}_j \in \D(\A)$.  As $\D(\A) \subset \H^2 \subset \mathbb{L}^p_\sigma(\Omega)$ for any $p \in [1, \infty)$, we have $\boldsymbol{\omega}_j \in \mathbb{L}_{\sigma}^{r+1}(\Omega)$. So from (2), we have 
\begin{align}
    \|\u^n - \u\|_{\mathrm{L}^{r+1}(0, T; \mathbb{L}^{r+1}_{\sigma}(\Omega))} \rightarrow 0, \text{ as } n \rightarrow \infty.
\end{align}
Let $\eta(t)$ be a positive, even, smooth function with compact support in $(-1, 1)$ such that
  \begin{align*}
      \int_{-\infty}^{\infty}\eta(t) \, dt = 1.
  \end{align*}
  We define a family of mollifiers related to $\eta$, denoted by $\eta_h$, as
  \begin{align}\label{mollifier}
      \eta_h(s) := \frac{1}{h}\eta\Big(\frac{s}{h}\Big), \quad h > 0.
  \end{align}
  For any element $\v \in \mathrm{L}^p(0, T; X),$ where $X$ is a Banach space and $p \in [1, \infty)$, we define the mollification of $\v$ in time as
  \begin{align*}
      \v_h(s) := (\v \ast \eta_h)(s) =  \int_0^T \v(\tau)\eta_h(s-\tau) \, d\tau, \quad h \in (0, T).
  \end{align*}
   This mollification has some important properties (cf. \cite{galdibook}) such as for any $\v \in \mathrm{L}^p(0, T; X), $ $ \v_h \in C^k([0, T]; X)$ for all $k\geq 0$ and 
   \begin{align*}
       \lim_{h \rightarrow 0} \|\v_h - \v\|_{\mathrm{L}^p(0, T; X)} = 0.
   \end{align*}
   Moreover, if $(\v^n_h)_{n=1}^{\infty}$ converges to $\v$ in $\mathrm{L}^p(0, T; X),$ then
   \begin{align*}
       \lim_{n \rightarrow \infty} \|\v^n_h - \v_h\|_{\mathrm{L}^p(0, T; X)} = 0.
   \end{align*}
   Using these properties of the mollifiers, we prove  the following result:
\begin{theorem}\label{thm4.10}
Every weak solution $(\u, \varphi)$ of the system \eqref{equ P}  on a bounded domain satisfies the energy equality:
\begin{align}\label{energy equality}
    &\frac{1}{2}\bigg(\|\u(t_1)\|^2 + \|\nabla\varphi(t_1)\|^2 + 2\int_{\Omega}F(\varphi)dx\bigg) + \nu \int_0^{t_1} \|\nabla\u(s)\|^2 \, ds + \int_0^{t_1}\|\u(s)\|_{\L_{\sigma}^{r+1}}^{r+1} ds \no\\&\quad + \int_0^{t_1} \|\sqrt{m(\varphi)}\nabla\mu(s)\|^2 ds 
   \nonumber\\& = \frac{1}{2}\bigg(\|\u_0\|^2 + \|\nabla\varphi_0\|^2 + 2\int_{\Omega}F(\varphi_0)dx\bigg)+ \int_0^{t_1}\langle\mathbb{U}(s), \u(s)\rangle \, ds, 
\end{align}
for $(\u_0, \varphi_0) \in \H \times \mathrm{H}^1,$ $r\geq 3$ and for every $t_1 \in [0, T]$.
\end{theorem}

\begin{proof}
As $\varphi$ and $\mu$ have enough regularity in the spatial variable, we only mollify $\mu$ in the time variable. In order to prove the energy equality \eqref{energy equality}, we follow the works \cite{robinson, Hajduk+Robinson_2017, galdibook,Kinra+Mohan_2024}.
  For some $t_1 > 0,$ we set 
  \begin{align}
  \mu_h(t) &= \int_0^T\mu(\tau)\chi_{[0, t_1]}(\tau)\eta_h(t-\tau) \, d\tau,\\
      \u^n_h(t) &= (\u^n \ast \eta_h)(t) = \int_0^{T}\u^n(\tau)\chi_{[0, t_1]}(\tau)\eta_h(t-\tau) \, d\tau,\label{mol u}
  \end{align}
  with the parameter $h>0$ satisfying $0< h < T-t_1 \text{ and } h<t_1$, where $\u^n$ is the eigenspace approximation of $\u$ defined in \eqref{eigensp approx} and $\eta_h$ is an even mollifier defined in \eqref{mollifier}. Since $\mu_h \in C^\infty(\Omega \times [0, T]), \, \u^n_h \in \mathbb{D}_{\sigma}(\Omega \times (0, T)),$ we can choose $ \v = \u^n_h\text{ and }\psi = \mu_h $ in  \eqref{weak form u}, \eqref{weak form phi}, respectively to get
  \begin{align}
      &-\int_0^{t_1}\langle \u(s), (\u^n_h)'(s)\rangle \, ds  + \nu \int_0^{t_1}\langle \nabla\u(s), \nabla\u^n_h(s)\rangle \, ds + \int_0^{t_1}\langle (\u(s)\cdot\nabla)\u(s), \u^n_h(s) \rangle \, ds \no\\
     &\quad+ \beta \int_0^{t_1}\langle |\u(s)|^{r-1}\u(s), \u^n_h(s) \rangle \, ds \no\\&= \int_0^{t_1}\langle \mu(s) \nabla \varphi(s), \u^n_h(s) \rangle \, ds,  -\langle \u(t), \u^n_h(t)\rangle + \langle\u(0), \u^n_h(0)\rangle,\label{test u}\\
     &\int_0^{t_1} \langle \varphi'(s), \mu_h(s)\rangle \, ds + \int_0^{t_1}\langle\boldsymbol{u}(s) \cdot \nabla \varphi(s), \mu_h(s)\rangle \, ds + \int_0^{t_1}\langle m(\varphi)\nabla \mu(s), \nabla\mu_h(s) \rangle \, ds  
     = 0.\label{test phi}
  \end{align}
We want to pass to limit as $n \rightarrow \infty$ in \eqref{test u}. 
An application of H\"older's inequality yields 
\begin{align}
  &\Big| \int_0^{t_1}\langle \u(s), (\u^n_h)'(s) - \u'_h(s)\rangle \, ds \Big|\leq \|\u\|_{\mathrm{L}^{\infty}(0,T;\H)}\|(\u^n_h)'- \u'_h\|_{\mathrm{L}^1(0,T;\H)}\to 0\ \text{ as }\ n\to\infty.
\end{align}
Similarly, we have 
\begin{align}\label{grad}
   \Big| \int_0^{t_1}\langle \nabla\u(s), \nabla\u^n_h(s)\rangle \, ds - \int_0^{t_1}\langle \nabla\u(s), \nabla\u_h(s)\rangle \, ds\Big|& \leq \|\u\|_{\mathrm{L}^2(0, T; \V_{\mathrm{div}})}\|\nabla(\u^n_h - \u_h)\|_{\mathrm{L}^2(0, T; \H)}\nonumber\\&\to 0\ \text{ as }\ n\to\infty. 
\end{align}
   Further, we deduce 
  \begin{align}\label{mu}
      \Big|&\int_0^{t_1}\langle \mu(s) \nabla \varphi(s), \u^n_h(s) \rangle \, ds - \int_0^{t_1}\langle \mu(s) \nabla \varphi(s), \u_h(s) \rangle \, ds \Big| \no \\  
     & = \Big|\int_0^{t_1}\langle \varphi(s)\nabla\mu(s), \u^n_h(s) \rangle \, ds - \int_0^{t_1}\langle \varphi(s)\nabla\mu(s), \u_h(s) \rangle \, ds \Big|\no\\
     & \leq C \int_{0}^{t_1}\|\nabla\mu(s)\| \|\varphi(s)\|_{\mathrm{L}^3}\|\u^n_h(s) - \u_h(s)\|_{\mathbb{L}^3_{\sigma}} \, ds \no\\
     & \leq C \|\nabla\mu\|_{\mathrm{L}^2(0, T; \mathbb{L}^2)}\|\nabla\varphi\|^{\frac{1}{2}}_{\mathrm{L}^\infty(0, T; \mathbb{L}^2)}\|\varphi\|^{\frac{1}{2}}_{\mathrm{L}^\infty(0, T; \mathrm{L}^2)}\|\u^n_h-\u_h\|_{\mathrm{L}^2(0, T; \V_{\mathrm{div}})} \no\\&\rightarrow 0\ \text{ as }\ n\to\infty. 
  \end{align}
Now, we pass to the limit in the convective term as follows:
  \begin{align}\label{trilinear}
      \Big|&\int_0^{t_1}\langle (\u(s)\cdot\nabla)\u(s), \u^n_h(s) \rangle \, ds - \int_0^{t_1}\langle (\u(s)\cdot\nabla)\u(s), \u_h(s) \rangle \, ds \Big| \no\\ 
      &\leq \|(\u\cdot\nabla)\u\|_{\mathrm{L}^2(0, T; \V'_{\mathrm{div}})}\|\u^n_h - \u_h\|_{{\mathrm{L}^2(0, T; \V_{\mathrm{div}})}} \rightarrow 0 \ \text{ as }\ n\to\infty,
  \end{align}
since $\|(\u\cdot\nabla)\u\|_{\mathrm{L}^2(0, T; \V'_{\mathrm{div}})}\leq \|\u\|_{\mathrm{L}^4(0,T;\mathbb{L}^4_{\sigma})}^2<\infty,$ for all $r\geq 3$.   Similarly, we estimate the damping term as 
  \begin{align}\label{absobtion}
      \Big|&\int_0^{t_1}\langle |\u(s)|^{r-1}\u(s), \u^n_h(s) \rangle \, ds - \int_0^{t_1}\langle |\u(s)|^{r-1}\u(s), \u_h(s) \rangle \, ds \Big| \no\\ &\leq \||\u|^{r-1}\u\|_{\mathrm{L}^{\frac{r+1}{r}}(0, T; \mathbb{L}_{\sigma}^{\frac{r+1}{r}})}\|\u^n_h - \u_h\|_{\mathrm{L}^{r+1}(0, T; \mathbb{L}^{r+1}_{\sigma})}\no\\
      & \leq \|\u\|_{\mathrm{L}^{r+1}(0, T; \mathbb{L}^{r+1}_{\sigma})}\|\u^n_h - \u_h\|_{\mathrm{L}^{r+1}(0, T; \mathbb{L}^{r+1}_{\sigma})} \rightarrow 0 \ \text{ as } \ n \rightarrow \infty.
  \end{align}
   So after passing to the limit as $n \rightarrow \infty$ in \eqref{test u}, we arrive at the following identity:
   \begin{align}\label{n tends infty}
     & -\int_0^{t_1}\langle \u(s), \u'_h(s)\rangle \, ds  + \nu \int_0^{t_1}(\nabla\u(s), \nabla\u_h(s)) \, ds + \int_0^{t_1}\langle (\u(s)\cdot\nabla)\u(s), \u_h(s) \rangle \, ds \no\\
     &\quad+ \beta \int_0^{t_1}\langle |\u(s)|^{r-1}\u(s), \u_h(s) \rangle \, ds \nonumber\\&= \int_0^{t_1}(\mu(s) \nabla \varphi(s), \u_h(s)) \, ds + \int_0^{t_1}\langle \mathbb{U}(s), \u_h(s)\rangle ds  - ( \u(t_1), \u_h(t_1)) + (\u(0), \u_h(0)).
   \end{align}
   Now we want to pass to the limit as $h \rightarrow 0$ in \eqref{n tends infty}. Note that since $\eta_h$ is an even function in $(-h, h)$, we have $\eta'_h(t) = -\eta'_h(-t)$. Hence, 
   \begin{align}\label{time mol}
       \int_0^{t_1}( \u(s), \u'_h(s) ) \, ds &= \int_0^{t_1} \Big( \u(s), \int_0^{t_1}\u(\tau)\eta'_h(s-\tau) \, d\tau \Big) \, ds \no\\
       & = \int_0^{t_1} \int_0^{t_1} \eta'_h(s-\tau)( \u(s), \u(\tau)) \, \d\tau \, ds \no\\
       & = - \int_0^{t_1} \int_0^{t_1} \eta'_h(\tau - s)(\u(s), \u(\tau)) \, \, d\tau \, ds\no\\
       & =  - \int_0^{t_1} \int_0^{t_1} \eta'_h(\tau - s)(\u(s), \u(\tau) ) \, ds \, \, d\tau = 0.
   \end{align}
   Now letting $h \rightarrow 0$ and arguing similarly as in \eqref{grad}, \eqref{mu}, \eqref{trilinear} and \eqref{absobtion}, we get
   \begin{align*}
       &\lim_{h \rightarrow 0}\int_0^{t_1} (\nabla\u(s), \nabla\u_h(s)) \, ds =\int_0^{t_1} (\nabla\u(s), \nabla\u(s)) \, ds, \\
       &\lim_{h \rightarrow 0}\int_0^{t_1} (\mu(s) \nabla \varphi(s), \u_h(s)) \, ds = \int_0^{t_1} (\mu(s) \nabla \varphi(s), \u(s)) \,ds,\\
       &\lim_{h \rightarrow 0}\int_0^{t_1}\langle (\u(s)\cdot\nabla)\u(s), \u_h(s) \rangle \, ds = \int_0^{t_1}\langle (\u(s)\cdot\nabla)\u(s), \u(s) \rangle \, ds,\\
       &\lim_{h \rightarrow 0}\int_0^{t_1} \langle|\u(s)|^{r-1}\u(s), \u_h(s) \rangle \, ds = \int_0^{t_1}\langle |\u(s)|^{r-1}\u(s), \u(s) \rangle \, ds,\\
       &\lim_{h \rightarrow 0}\int_0^{t_1}\langle\mathbb{U}(s), \u_h(s)\rangle ds = \int_0^{t_1}\langle\mathbb{U}(s), \u(s)\rangle ds,
   \end{align*}
   which gives 
   \begin{align}\label{h lim}
     \nu \int_0^{t_1} \|\nabla\u(s)\|^2 \, ds  + \beta \int_0^{t_1}\|\u(s)\|^{r+1}_{\mathbb{L}_{\sigma}^{r+1}(\Omega)} ds= & \int_0^{t_1}\langle \mu(s) \nabla \varphi(s), \u(s) \rangle \,ds + \int_0^{t_1}\langle\mathbb{U}(s), \u(s)\rangle \, ds  \no\\
     & - \lim_{h \rightarrow 0}(\u(t_1), \u_h(t_1))+ \lim_{h \rightarrow 0} (\u(0), \u_h(0) ),  
   \end{align}
 since $\u\in C_w([0,T];\H).$  We evaluate the last two terms of \eqref{h lim} as follows:
   \begin{align*}
        (\u(t_1), \u_h(t_1)) =& \Big( \u(t_1), \int_0^{t_1}\u(s)\eta_h(t_1-s) \, ds \Big)\\
       =& \int_0^{t_1} \eta_h(s)( \u(t_1), \u(t_1-s)) \, ds \\
       =&\int_0^h \eta_h(s) ( \u(t_1), \u(t_1-s) ) \, ds\\
       = &\int_0^h\eta_h(s)\|\u(t_1)\|^2 \, ds + \int_0^h \eta_h(s) ( \u(t_1), \u(t_1-s)-\u(t_1)) \, ds .
   \end{align*} 
   Using the $\H$-weak continuity of $\u$ and the fact that $\int_0^h \eta_h(s) \, ds = \frac{1}{2}$, we obtain
\begin{align*}
    \lim_{h \rightarrow 0}( \u(t_1), \u_h(t_1) ) = \frac{1}{2}\|\u(t_1)\|^2.
\end{align*}
Similarly, we have 
\begin{align*}
     \lim\limits_{h \rightarrow 0} (\u(0), \u_h(0) ) = \frac{1}{2}\|\u(0)\|^2.
\end{align*}
Finally, from \eqref{h lim}, we derive
\begin{align}\label{ee u}
    &\frac{1}{2}\|\u(t_1)\|^2 + \nu \int_0^{t_1} \|\nabla\u(s)\|^2 \, ds  + \beta \int_0^{t_1}\|\u(s)\|^{r+1}_{\mathbb{L}_{\sigma}^{r+1}(\Omega)} \nonumber\\&=  \; \frac{1}{2}\|\u(0)\|^2 + \int_0^{t_1}\langle \mu(s) \nabla \varphi(s), \u(s) \rangle \,ds  + \int_0^{t_1}\langle\mathbb{U}(s), \u(s)\rangle ds,
\end{align}
for all $t_1 \in [0, T].$

Let us now pass to limit as $h \rightarrow 0$ in \eqref{test phi}. Consider 
\begin{align}
    \int_0^{t_1} \langle \varphi'(s), \mu_h(s)\rangle \, ds &= -\int_0^{t_1}\langle \varphi'(s), (\Delta\varphi)_h(s)\rangle \, ds+ \int_0^{t_1}\langle \varphi'(s), (F'(\varphi))_h(s)\rangle \, ds\no\\
    &:= I_1 + I_2.\no
\end{align}
We can write $I_1$ as
\begin{align}
   I_1 &=  \int_0^{t_1}\langle\varphi(s),(\Delta\varphi)'_h(s)\rangle ds  -\langle\varphi(t_1),(\Delta\varphi)_h(t_1)\rangle + \langle\varphi(0),(\Delta\varphi)_h(0)\rangle.\no
\end{align}
Proceeding similarly as in \eqref{time mol}, we obtain 
\begin{align}
    \int_0^{t_1}\langle\varphi(s),(\Delta\varphi)'_h(s)\rangle ds  = 0.
\end{align}
We can pass to the limit in the second term of $I_1$ as follows:
\begin{align}
  \lim_{h\rightarrow 0} -\langle\varphi(t_1),(\Delta\varphi)_h(t_1)\rangle &= \lim_{h\rightarrow 0}-\left\langle\varphi(t_1), \int_0^{t_1}\eta_h(t_1-s)\Delta\varphi(t_1) \, ds \right\rangle\no\\
     & = \lim_{h\rightarrow 0}- \int_0^{h}\eta_h(t_1-s)\langle\varphi(t_1), \Delta\varphi(t_1)\rangle \, ds = \frac{1}{2}\|\nabla\varphi(t_1)\|^2,
\end{align}
where we have used the fact that $\eta_h$ is compactly supported in $(-h, h)$, $h<t_1$ and $\int_0^h\eta_h(s) \, ds = \frac{1}{2}$. Similarly, we can pass to the limit in the last term of $I_1$. Therefore, we deduce 
\begin{align}\label{lim I1}
    \lim_{h \rightarrow 0}I_1= \frac{1}{2}\|\nabla\varphi(t_1)\|^2- \frac{1}{2}\|\nabla\varphi(0)\|^2.
\end{align}
Next, we observe that 
\begin{align}\label{I_2}
     &\Big|\int_0^{t_1}\langle\varphi'(s), (F'(\varphi))_h(s) - (F'(\varphi(s))\rangle \, ds\Big|\no\\
    & \leq \|\varphi'\|_{\mathrm{L}^2(0, T; (\mathrm{H}^1)')}\|(F'(\varphi))_h - (F'(\varphi))\|_{\mathrm{L}^2(0, T; \mathrm{H}^1)}.
\end{align}
It can be seen that $F'(\varphi) \in \mathrm{L}^2(0, T; \mathrm{H}^1)$, since  by using  Gagliardo-Nirenberg inequality, we obtain 
\begin{align}\label{F'calculation}
   \|F'(\varphi)\|_{\mathrm{L}^2(0, T;\mathrm{H}^1)}^2 &= \|F'(\varphi)\|_{\mathrm{L}^2(0, T; \mathrm{L}^2)}^2 + \|\nabla F'(\varphi)\|_{\mathrm{L}^2(0, T; \mathbb{L}^2)}^2\no\\
   &\leq  2 C_1^2 \int_0^T\|\varphi(t)\|^{2p}_{\mathrm{L}^{2p}} dt +2C_2^2|\Omega|T\nonumber\\&\quad+2 C_3^2  \|\nabla\varphi\|_{\mathrm{L}^\infty(0, T; \mathbb{L}^2)}\left(|\Omega|T+\int_0^T\|\varphi(t)\|_{\mathrm{L}^{\infty}}^{2(p-1)} \, dt\right) := S.
\end{align}
Note that for $1\leq p\leq 3$, $S$ is finite. For $3<p\leq 4,$ we can estimate $S$ as follows:
\begin{align}
 S \leq C\left(1+\int_0^T\|\varphi(t)\|_{\mathrm{H}^2}^{p-3}\|\varphi(t)\|^{p+3}_{\mathrm{H}^1} \, dt+\int_0^T\|\varphi(t)\|^{p-1}_{\mathrm{H}^1}\|\varphi(t)\|^{p-1}_{\mathrm{H}^2} \, dt\right) < \infty.
\end{align} 
So passing to the limit as $h \rightarrow 0$ in \eqref{I_2}, we have 
\begin{align}\label{lim I_2}
    \lim_{h \rightarrow 0}I_2 = \int_0^{t_1}\langle\varphi'(s), F'(\varphi(s))\rangle \, ds.
\end{align}
Now as explained in \cite[proof of Corollary 2]{colli_weak},  one can deduce that 
\begin{align*}
    \langle\varphi', F'(\varphi)\rangle = \frac{d}{dt}\int_\Omega F(\varphi)\, dx.
\end{align*}
Therefore, from \eqref{lim I_2}, we get
\begin{align}\label{limI2}
     \lim_{h \rightarrow 0}I_2 = \int_\Omega F(\varphi(t_1)) \, d x- \int_\Omega F(\varphi(0))\, d x.
\end{align}
We can easily calculate the limit of the second and third terms of \eqref{test phi} as below:
\begin{align}\label{2ndtestphi}
    &\bigg|\int_0^{t_1}\langle\boldsymbol{u}(s) \cdot \nabla \varphi(s), \mu_h(s)\rangle \, ds - \int_0^{t_1}\langle\boldsymbol{u}(s) \cdot \nabla \varphi(s), \mu(s)\rangle \, ds\bigg|\no\\ &= \bigg|\int_0^{t_1}\langle\boldsymbol{u}(s) \cdot \nabla \varphi(s), \mu_h(s)-\mu(s)\rangle \, ds\bigg| \leq \int_0^{t_1}\|\u(s)\|_{\mathbb{L}^3_{\sigma}}\|\nabla\varphi(s)\|_{\mathrm{L}^2}\|\mu_h(s)-\mu(s)\|_{\mathrm{L}^6} ds\no\\
    &\leq \int_0^{t_1}\|\u(s)\|_{\H}^{\frac{1}{2}}\|\u(s)\|^{\frac{1}{2}}_{\V_{\mathrm{div}}}\|\nabla\varphi(s)\|_{\mathbb{L}^2}\|\mu_h(s)-\mu(s)\|_{\mathrm{H}^1} ds \no\\
    &\leq \sup_{s\in[0, T]}\|\nabla\varphi(s)\|_{\mathbb{L}^2}\Big(\int_0^{t_1}\|\u(s)\|_{\H}\|\u(s)\|_{\V_{\mathrm{div}}} ds\Big)^{\frac{1}{2}}\Big(\int_0^{t_1}\|\mu_h(s)-\mu(s)\|^2_{\mathrm{H}^1} ds\Big)^{\frac{1}{2}} \nonumber\\&\rightarrow 0 \ \text{ as }\  h \rightarrow 0,
\end{align}
and 
\begin{align}\label{3rdtesrphi}
    &\bigg|\int_0^{t_1}\langle m(\varphi(s))\nabla \mu(s), \nabla\mu_h(s)-\mu(s) \rangle \, ds\bigg| \nonumber\\&\leq m_2 \Big(\int_0^{t_1}\|\nabla\mu(s)\|^2 ds\Big)^{\frac{1}{2}}\Big(\int_0^{t_1}\|\nabla\mu_h(s)-\nabla\mu(s)\|^2 ds \Big)^{\frac{1}{2}}\no\\
    & \rightarrow 0 \ \text{ as }\  h \rightarrow 0.
\end{align}
Therefore passing to the limit in \eqref{test phi} with the help of \eqref{lim I1}, \eqref{limI2}, \eqref{2ndtestphi}, and \eqref{3rdtesrphi} yield
\begin{align}\label{limtestphi}
   & \frac{1}{2}\|\nabla\varphi(t_1)\|^2 + \int_\Omega F(\varphi(t_1))\, dx + \int_0^{t_1}\langle\u(s)\cdot\nabla\varphi(s), \mu(s)\rangle ds + \int_0^{t_1}\|\sqrt{m(\varphi(s))}\nabla\mu(s)\|^2 ds \no\\&= \frac{1}{2}\|\nabla\varphi(0)\|^2 + \int_\Omega F(\varphi(0))\, dx.
    \end{align}
 Now adding \eqref{ee u} with \eqref{limtestphi}, we finally deduce  the energy equality for any weak solution $(\u, \varphi).$
\end{proof}
\begin{remark}
   From the energy equality \eqref{energy equality}, it is immediate that any weak solution of \eqref{equ P} ($r\geq 3$) is continuous in time in $\H\times\mathrm{H}^1$, that is, $(\u,\varphi)\in C([0,T];\H\times\mathrm{H}^1)$. 
\end{remark}
\begin{remark}
    It should be noted that this type of energy equality or continuity of weak solutions in $H\times\mathrm{H}^1$  is not proved for the Cahn-Hilliard-Navier-Stokes system in three dimensions till now. This has been studied under Serrin-type condition $\u\in \mathrm{L}^r(0,T;\mathbb{L}^s_{\sigma})$ for $\frac{2}{r}+\frac{3}{s}=1$. But adding an absorption term in the Cahn-Hilliard-Navier-Stokes system provides us the continuity  result for $\u\in \mathrm{L}^4(0,T;\mathbb{L}^4_{\sigma})$.
\end{remark}

\section{Uniqueness}\label{sec4}
In this section, we prove the uniqueness of weak solutions and also the continuous dependence on the data under the assumption that mobility is a constant for $r \geq 3$ and $(\u_0, \varphi_0) \in \H \times\mathrm{H}^2$.  For  $r=3,$ the same  results hold true when  fluid's viscosity ($\nu$) and medium's porosity ($\beta$) are sufficiently large ($\beta\nu\geq 1$). However to the end of this section, we employ a new technique using the energy equality satisfied by  weak solutions to prove the uniqueness for the case $r=3$ without  the extra condition on $\nu$ and $ \beta$. Thus we have the uniqueness result for all $r \geq 3$ and $\nu,\beta>0$.
\begin{theorem}\label{unique}
 Let $(\u, \varphi)$ be a weak solution of \eqref{equ P} with initial condition $(\u_0, \varphi_0) \in \H \times\mathrm{H}^2$ and mobility $m=1.$ Then the solution is unique for every $r>3$ and for $r=3,$ the solution is unique if $\beta\nu \geq 1.$   
\end{theorem}
\begin{proof}(\textbf{Case I: $\boldsymbol{r > 3}$}.)
Let $(\u_1, \varphi_1) \text{ and } (\u_2, \varphi_2)$ be two pairs of weak solutions of the system \eqref{equ P} with initial datum $(\u_{01}, \varphi_{01}) , \, (\u_{02}, \varphi_{02}) \in \H\times\mathrm{H}^2$ and forcing $\mathbb{U}_1$ and $\mathbb{U}_2,$ respectively. Since $(\u_{0i}, \varphi_{0i})\in \H \times\mathrm{H}^2,$ one has the regularity given by Lemma \ref{lem4.9}.  Moreover, by Theorem \ref{thm4.10}, weak solutions $(\u_{i}, \varphi_{i})$ satisfy the energy equality also.    Let us define $\z:= \u_1 - \u_2, \, \mathbb{U}= \mathbb{U}_1-\mathbb{U}_2, \, \rho:= \varphi_1 - \varphi_2, \text{ and } \mu:= \mu_1 - \mu_2$. Then $(\z, \rho)$ satisfy the following for all $(\v,\psi)\in \V_{\mathrm{div}}\times\mathrm{H}^1(\Omega)$:
   \begin{align}
     &\langle \boldsymbol{z}', \v \rangle + \nu \langle \nabla\boldsymbol{z}, \nabla\v \rangle + \langle (\boldsymbol{u}_1\cdot \nabla) \boldsymbol{z}, \v \rangle + \langle (\z \cdot\nabla)\u_2, \v \rangle + \beta \langle |\u_1|^{r-1}\u_1 - |\u_2|^{r-1}\u_2, \v \rangle \no\\& \quad= \langle \mu\nabla\varphi_1, \v \rangle + \langle \mu_2\nabla\rho, \v \rangle + \langle\mathbb{U}, \v\rangle,\label{weak form u 1}\\
     &\langle \rho', \psi\rangle +  \langle \z \cdot \nabla\varphi_1, \psi \rangle + \langle\u_2\cdot\nabla\rho, \psi \rangle + \langle \nabla \mu, \nabla\psi \rangle = 0. \label{weak form phi 2}
   \end{align}
   Testing the equations \eqref{weak form u 1} with $\z$ and \eqref{weak form phi 2} with $\rho + \mathcal{B}^{-1}\rho$, we get
   \begin{align}
     \frac{1}{2}\frac{d}{dt}\|\z\|^2 + \nu\|\nabla\z\|^2 &= -\langle (\z \cdot\nabla)\u_2, \z \rangle -\beta\langle |\u_1|^{r-1}\u_1 - |\u_2|^{r-1}\u_2, \z \rangle + \langle \mu\nabla\varphi_1, \z \rangle \no\\
     &\quad+ \langle \mu_2\nabla\rho, \z \rangle  +\langle\mathbb{U}, \z\rangle =: \sum_{i=1}^{5}I_i, \label{difference equ u}
    \end{align}
    and
    \begin{align}
     \frac{1}{2}\frac{d}{dt}(\|\rho\|_{\ast}^2 + \|\rho\|^2) + (\mu, \rho)&= -( \z \cdot \nabla\varphi_1, \rho) - ( \nabla\mu, \nabla\rho ) - ( \z \cdot \nabla\varphi_1, \mathcal{B}^{-1}\rho ) - (\u_2\cdot\nabla\rho, \mathcal{B}^{-1}\rho )\no\\
     & = \sum_{i=1}^{4}J_i,\label{diffrence equ phi}
   \end{align}
   where $\|\rho\|_{\ast} = \|\nabla \mathcal{B}^{-1}\rho\|$. By the Assumption \ref{prop of F} (2) on $F,$ we have 
   \begin{align*}
       (\mu, \rho)= (-\Delta\rho + F'(\varphi_1)-F'(\varphi_2), \rho) \geq \|\nabla\rho\|^2 -C_0\|\rho\|^2.
   \end{align*}
   From the definition $\mathcal{B}^{-1}$ given in \eqref{bes1}, we get
   \begin{align*}
       C_0\|\rho\|^2 = C_0(\nabla\mathcal{B}^{-1}\rho, \nabla\rho) \leq \frac{1}{4}\|\nabla\rho\|^2 + C_0^2\|\rho\|_\ast.
   \end{align*}
   Then, using \eqref{diffrence equ phi} we end up with the following:
   \begin{align}
        \frac{1}{2}\frac{d}{dt}(\|\rho\|_{\ast}^2 + \|\rho\|^2) +\frac{3}{4}\|\nabla\rho\|^2 \leq C_0^2\|\rho\|^2_\ast +\sum_{i=1}^{4}J_i.\label{diffrence equ phi1}
   \end{align}
   We can estimate $I_1, \, I_2$ as discussed  in \cite{MTMohan,MTMohan1}. Special attention has been given for estimating the terms $I_3, \, I_4$. The following inequalities describe the bounds. For $I_1, \, I_2,$ we have
   \begin{align*}
       |I_1| & \leq \frac{\nu}{4}\|\nabla\z\|^2 + \frac{\beta}{2}\||\u_2|^{\frac{r-1}{2}}\z\|^2 + \frac{r-3}{2\nu(r-1)}\Big(\frac{4}{\beta\nu(r-1)}\Big)^{\frac{2}{r-3}}\|\z\|^2,\\
       I_2 & \geq \frac{\beta}{2}\||\u_2|^{\frac{r-1}{2}}\z\|^2.
   \end{align*}
   To estimate $I_3, \, I_4,$ we use a generalized Gagliardo-Nirenberg inequality, Sobolev inequality, and Young's inequality as follows:
   \begin{align*}
       |I_3| & \leq |( \Delta\rho\nabla\varphi_1, \z )| + |( (F'(\varphi_1) - F'(\varphi_2))\nabla\varphi_1, \z )|\no\\
       & \leq \|\Delta\rho\| \|\nabla\varphi_1\|_{\mathbb{L}^6} \|\z\|_{\mathbb{L}^3_{\sigma}} + (\|\varphi_1\|^{p-1}_{\mathrm{L}^\infty}+ \|\varphi_2\|^{p-1}_{\mathrm{L}^\infty}) \|\rho\|\|\nabla\varphi_1\|_{\mathbb{L}^3} \|\z\|_{\mathbb{L}^3_{\sigma}}\no\\
       & \leq \frac{1}{2}\|\Delta\rho\|^2 + \frac{\nu}{12}\|\nabla\z\|^2 + C \|\varphi_1\|_{\mathrm{H}^2}^4\|\z\|^2 + C (\|\varphi_1\|^{2p-2}_{\mathrm{L}^\infty}+ \|\varphi_2\|^{2p-2}_{\mathrm{L}^\infty}) \|\nabla\varphi_1\|^2_{\mathbb{L}^3}\|\rho\|^2 \no\\
       & \leq \frac{1}{2}\|\Delta\rho\|^2 + \frac{\nu}{12}\|\nabla\z\|^2 + C \|\varphi_1\|_{\mathrm{H}^2}^4\|\z\|^2  + C \big(\|\varphi_1\|^{p-1}_{\mathrm{H}^1}\|\varphi_1\|^{p-1} _{\mathrm{H}^2} \no\\ & \quad + \|\varphi_2\|^{p-1}_{\mathrm{H}^1} \|\varphi_2\|^{p-1}_{\mathrm{H}^2}\big) \|\varphi_1\|_{\mathrm{H}^1}\|\varphi_1\|_{\mathrm{H}^2}\|\rho\|^2,\\
       |I_4| & \leq \|\mu_2\|_{\mathrm{L}^6}\|\nabla\rho\|\|\z\|_{\mathbb{L}^3_{\sigma}} \leq \frac{1}{2}\|\nabla\rho\|^2+\frac{1}{2}\|\mu_2\|^2_{\mathrm{L}^6}\|\z\|\|\nabla\z\|\no\\
       & \leq \frac{1}{2}\|\nabla\rho\|^2 + \frac{\nu}{12}\|\nabla\z\|^2 + C\|\mu_2\|^4_{\mathrm{L}^6}\|\z\|^2 \no\\
       &\leq \frac{1}{2}\|\nabla\rho\|^2 + \frac{\nu}{12}\|\nabla\z\|^2 + C\|\mu_2\|^4_{\mathrm{H}^1}\|\z\|^2. \\ 
  \text{and} \;    |I_5| &\leq \frac{1}{2} \| \mathbb{U} \|^2 + \frac{1}{2} \|\z \|^2.
  \end{align*}
   Similarly, we can estimate $J_1$ using a generalized Gagliardo-Nirenberg inequality and Young's inequality as
   \begin{align*}
     |J_1| & \leq \|\rho\|\|\z\|_{\mathbb{L}^6_{\sigma}}\|\nabla\varphi_1\|_{\mathbb{L}^3} 
      \leq \frac{\nu}{12}\|\nabla\z\|^2 + C\|\varphi_1\|_{\mathrm{H}^1}\|\varphi_1\|_{\mathrm{H}^2}\|\rho\|^2.
   \end{align*}
   Moreover, using the assumption on $F$,  we estimate $J_2$ as follows:
   \begin{align*}
       J_2 & = -\|\Delta\rho\|^2 - (F'(\varphi_1) - F'(\varphi_2), \Delta\rho)\leq -\frac{3}{4}\|\Delta\rho\|^2 + C_0\|\rho\|^2.
   \end{align*}
   Following \cite{uniqueness}, we can estimate $J_3,  \text{ and } J_4$ as
   \begin{align*}
       |J_3| & \leq \|\varphi_2\|_{\mathrm{L}^\infty}\|\z\|\|\rho\|_\ast \leq \frac{1}{2}\|\z\|^2 + \frac{1}{2}\|\varphi_1\|_{\mathrm{H}^1}\|\varphi_1\|_{\mathrm{H}^2}\|\rho\|^2_{\ast},\\
       |J_4| &  \leq \|\rho\|_{\mathrm{L}^6}\|\u_2\|_{\mathbb{L}^3_{\sigma}}\|\rho\|_\ast \leq \frac{1}{8}\|\nabla\rho\|^2 + C\|\u_2\|^2_{\mathbb{L}^3_{\sigma}}\|\rho\|^2_{\ast}.
   \end{align*} 
   Taking all the estimates of $I_i$ and $J_i$ into account of \eqref{difference equ u} and \eqref{diffrence equ phi} and adding them together, we finally arrive at 
   \begin{align*}
       &\frac{1}{2}\frac{d}{dt}(\|\z\|^2 + \|\rho\|_{\ast}^2 + \|\rho\|^2) +  \frac{\nu}{4}\|\nabla\z\|^2 + \frac{1}{8}\|\nabla\rho\|^2 + \frac{1}{4}\|\Delta\rho\|^2  \no \\ & \leq C\|\mathbb{U}\|^2_{\V'_{\mathrm{div}}} +\Big(C\|\u_1\|^2_{\mathbb{L}^3_{\sigma}} + \frac{1}{2} \|\varphi_1\|_{\mathrm{H}^1} \|\varphi_1\|_{\mathrm{H}^2}\Big) \|\rho\|^2_{\ast}+ \Big( C\big(\|\varphi_1\|^{p-1}_{\mathrm{H}^1}\|\varphi_1\|^{p-1} _{\mathrm{H}^2}  \no\\ & \quad + \|\varphi_2\|^{p-1}_{\mathrm{H}^1} \|\varphi_2\|^{p-1}_{\mathrm{H}^2}\big) \|\varphi_1\|^2_{\mathrm{H}^1} \|\varphi_1\|_{\mathrm{H}^2} + C\|\varphi_1\|_{\mathrm{H}^1} \|\varphi_1\|_{\mathrm{H}^2} + C_0\Big)\|\rho\|^2  \no\\ & \quad+ \Big( \frac{r-3}{4\nu(r-1)}\Big(\frac{4}{\beta\nu(r-1)}\Big)^{\frac{2}{r-3}} + C\|\varphi_1\|^4_{\mathrm{H}^2} + C\|\mu_2\|^4_{\mathrm{H}^1} + \frac{1}{2}\Big)\|\z\|^2.
   \end{align*}
   Integrating the above inequality between $0$ to $t$ for any $t\in [0, T]$ and applying Gronwall's inequality, we conclude that
   \begin{align*}
   \|\z\|_{\mathrm{L}^\infty(0, T; \H)\cap \mathrm{L}^2(0, T; \V_{\mathrm{div}})} + \|\rho\|_{\mathrm{L}^\infty(0, T; D(\mathcal{B}^{-1})\cap \mathrm{L}^2) \cap \mathrm{L}^2(0, T; \mathrm{H}^2)} \leq C(\|\z_0\|^2 + \|\rho_0\|^2_{\ast}+ \|\rho_0\|^2).
   \end{align*}
   
   (\textbf{Case II: $\boldsymbol{r > 3}$}.)
   For the uniqueness of the weak solution of the system \eqref{equ P}, we  only have to change the estimate of $I_1$ of \eqref{difference equ u}. So we estimate $I_1$ as follows:
 \begin{align*}
     |I_1| & \leq \frac{\nu}{2}\|\nabla\z\|^2 + \frac{1}{2\nu}\|\u_2\z\|^2.
 \end{align*}
 All the other estimates of $I_i \text{ and } J_i$ follow similarly as in the case of  $r>3 $. Therefore, we get
 \begin{align*}
     &\frac{1}{2}\frac{d}{dt}(\|\z\|^2 + \|\rho\|_{\ast}^2 + \|\rho\|^2) + \frac{1}{2}(\beta - \frac{1}{\nu})\|\u_2\z\|^2 + \frac{1}{4}\|\Delta\rho\|^2 \no\\ & \leq C\|\mathbb{U}\|_{\V'_{\mathrm{div}}} + \Big(C\|\u_1\|^2_{\mathbb{L}^3_{\sigma}} + \frac{1}{2} \|\varphi_1\|_{\mathrm{H}^1} \|\varphi_1\|_{\mathrm{H}^2}\Big) \|\rho\|^2_{\ast} 
        + \Big( C\big(\|\varphi_1\|^{p-1}_{\mathrm{H}^1}\|\varphi_1\|^{p-1} _{\mathrm{H}^2}  \no\\ & \quad + \|\varphi_2\|^{p-1}_{\mathrm{H}^1} \|\varphi_2\|^{p-1}_{\mathrm{H}^2}\big) \|\varphi_1\|^2_{\mathrm{H}^1} \|\varphi_1\|_{\mathrm{H}^2} + C\|\varphi_1\|_{\mathrm{H}^1} \|\varphi_1\|_{\mathrm{H}^2} + C_0\Big)\|\rho\|^2 + \Big(C\|\mu_2\|^4_{\mathrm{H}^1} + \frac{1}{2}\Big)\|\z\|^2.
 \end{align*}
 If $\beta\nu \geq 1$, using Gronwall's inequality, we can conclude uniqueness for the case $r=3$ as the case of $r> 3$.
\end{proof}
\begin{remark}
	In fact, for $3<p\leq 6$, if estimate the term $I_4$ as
	\begin{align*}
	 |I_4| & \leq \|\mu_2\|_{\mathrm{L}^p}\|\nabla\rho\|\|\z\|_{\mathbb{L}^{\frac{2p}{p-2}}_{\sigma}} \leq \frac{1}{2}\|\nabla\rho\|^2+\frac{1}{2}\|\mu_2\|^2_{\mathrm{L}^p}\|\z\|^{2-\frac{6}{p}}\|\nabla\z\|^{\frac{6}{p}}\no\\
	& \leq \frac{1}{2}\|\nabla\rho\|^2 + \frac{\nu}{12}\|\nabla\z\|^2 + C\|\mu_2\|^{\frac{2p}{p-3}}_{\mathrm{L}^p}\|\z\|^2,
	\end{align*}
	then the required condition on $\mu_2$ to be $\int_0^T\|\mu_2(s)\|^{\frac{2p}{p-3}}_{\mathrm{L}^p}ds<\infty$. 
\end{remark}
Now, we introduce a novel approach to prove uniqueness of the weak solution of the system \eqref{equ P} for the case of $r=3$, without any condition on the parameters like  viscosity and absorption coefficient. Moreover, the result given below establishes uniqueness result for $(\u_0, \varphi_0) \in \H \times\mathrm{H}^1$. We exploit the energy equality satisfied by weak solution to obtain the required result. Let us discuss the idea briefly. 
\begin{theorem}\label{uni}
Let $(\u, \varphi)$ be a weak solution of \eqref{equ P} with initial condition $(\u_0, \varphi_0) \in \H \times\mathrm{H}^1$ and mobility $m=1.$ Then weak solution of the system \eqref{equ P} is unique for $r\geq3$.
    \end{theorem}
    
    \begin{proof}
    Let $(\u, \varphi)$ be a weak solution of the problem \eqref{equ P} satisfying the energy equality \eqref{energy equality} with the initial data $(\u_0,\varphi_0)$. Now we consider the following system:
    \begin{equation}\label{equ}
\left\{
\begin{aligned}
\boldsymbol{v}' - \nu \Delta\boldsymbol{v} + (\boldsymbol{u}\cdot \nabla)\boldsymbol{v} + \beta|\v|^{r-1}\v + \nabla p &= \hat{\mu} \nabla \varphi + \mathbb{U}, \, \, \text{ in } \Omega \times (0,T), \\
  \psi' + \boldsymbol{v} \cdot \nabla \varphi &= \Delta\hat{\mu}, \, \, \text{ in } \Omega \times (0,T), \\
        \hat{\mu} &= -\Delta\psi + F'(\varphi), \\
        \mathrm{div}~\boldsymbol{v} & = 0, \, \, \text{ in } \Omega \times (0,T), \\
        \frac{\partial\psi}{\partial\n} = 0, \,  \frac{\partial\hat{\mu}}{\partial\boldsymbol{n}} & = 0, \,\, \text{ on } \partial\Omega\times(0,T), \\
        \v & = \boldsymbol{0}, \,\, \text{ on } \partial\Omega\times(0,T), \\
        \boldsymbol{v}(0) = \boldsymbol{u}_0 ,\,\, \psi(0) & = \varphi_0, \,\, \text{ in } \Omega,  
\end{aligned}   
\right.
\end{equation}
 where $(\u,\varphi)$ as described above. As the convective term is linear and its regularity is known a-priori (cf. Theorem \ref{LH weak sol} and Definition \ref{weak sol defn}),  proceeding similarly as in Theorem \ref{LH weak sol}, we can show the existence of a Leray-Hopf weak solution  $(\v, \psi)$  of the problem \eqref{equ} such that 
 \begin{equation}\label{exist vpsi}
\left\{
\begin{aligned}
     \v &\in  \mathrm{L}^{\infty}(0, T; \H) \cap \mathrm{L}^2(0, T; \V_{\mathrm{div}}) \cap \mathrm{L}^{r+1}(0, T; \mathbb{L}^{r+1}_{\sigma}),\\
       \psi &\in \mathrm{L}^{\infty}(0, T; \mathrm{H}^1)\cap \mathrm{L}^2(0, T; \mathrm{H}^2),\\
      \tilde  \mu &\in  \mathrm{L}^{2}(0, T; \mathrm{H}^1) \\
\v' &\in \mathrm{L}^2(0, T; \V_{\mathrm{div}}') + \mathrm{L}^{\frac{r+1}{r}}(0, T; \mathbb{L}_\sigma^{\frac{r+1}{r}})\subset \mathrm{L}^{\frac{r+1}{r}}(0, T;\V_{\mathrm{div}}'+ \mathbb{L}_{\sigma}^{\frac{4}{3}}),\\
\psi' & \in \mathrm{L}^2(0, T, (\mathrm{H}^1)') 
 \end{aligned}
 \right.
 \end{equation}
 We claim that the solution $(\v, \psi)$   is also unique. Let $(\v_i, \psi_i), \, i=1,2$ be two solutions with the same initial data $(\u_0,\varphi_0)$ and we set $\Tilde{\v}=\v_1-\v_2, \, \Tilde{\psi}= \psi_1-\psi_2$ and $\Tilde{\mu}= \hat{\mu}_1-\hat{\mu}_2.$ We consider the difference equation satisfied by $(\Tilde{\v}, \Tilde{\psi})$ which is 
\begin{equation}
\left\{
\begin{aligned}
&\Tilde{\v}'-\nu\Delta\Tilde{\v}+(\u\cdot\nabla)\Tilde{\v} + \beta(|\v_1|^{r-1}\v_1-|\v_2|^{r-1}\v_2) +\nabla\Tilde{p}= \Tilde{\mu}\nabla\varphi, \quad \text{ in  }\Omega\times(0, T),\\
&\Tilde{\psi}' +\v_1\cdot\nabla\varphi-\v_2\cdot\nabla\varphi=\Delta\Tilde{\mu}, \quad \text{ in  }\Omega\times(0, T),\\
& \mathrm{div}~\boldsymbol{v}  = 0, \, \, \text{ in } \Omega \times (0,T), \\
&\boldsymbol{\tilde v}   = 0,  \; \frac{\partial \psi}{\partial n} = 0,  \; \frac{\partial \Delta \psi}{\partial n}  = 0  \ \text{ in } \partial \Omega \times (0, T),
 \end{aligned}
 \right.
 \end{equation}
 and initial condition $$\Tilde{\v}(0)=0, \quad \Tilde{\psi}(0)=0, \quad \text{in }\Omega.$$
 Then using similar mollification technique as in Theorem \ref{thm4.10}, one  can show $(\Tilde{\v}, \Tilde{\psi})$ satisfies the following energy equality:
 \begin{align}\label{diff v}
     \|\Tilde{\v}(t)\|^2 &+ \|\nabla\Tilde{\psi}(t)\|^2 +2\nu\int_0^t\|\nabla\Tilde{\v}(s)\|^2ds+ 2\int_0^t\|\nabla\Tilde{\mu}(s)\|^2ds \nonumber\\&+ 2\beta\int_0^t\langle|\v_1(s)|^{r-1}\v_1(s)-|\v_2(s)|^{r-1}\v_2(s), \v_1(s)-\v_2(s)\rangle ds = 0,
 \end{align}
for all $t\in[0,T]$. Observing that 
 \begin{align}\label{410}
     \langle|\v_1|^{r-1}\v_1-|\v_2|^{r-1}\v_2, \v_1-\v_2\rangle \geq  0, \quad \ \mbox{ for all }\ \v_1, \v_2 \in \V_{\mathrm{div}}\cap\mathbb{L}^4_\sigma, \text{ and } r\geq 1,
 \end{align}
 we have from \eqref{diff v} that  $\v_1=\v_2$ and $\psi_1=\psi_2$  in $\mathbb{H}\times \mathrm{L}^2$ for all $t\in[0,T]$. Now we consider the difference of the equations \eqref{equ P} and \eqref{equ}.  We define  $\w = \u-\v, \, \rho = \varphi-\psi$, so that $(\w, \rho)$ satisfies the following:
 \begin{equation}\label{diff wrho}
     \left \{
     \begin{aligned}
         &\boldsymbol{w}' - \nu \Delta\boldsymbol{w} + (\boldsymbol{u}\cdot \nabla)\boldsymbol{w} + \beta(|\u|^{r-1}\u - |\v|^{r-1}\v)+ \nabla q = -\Delta^2\rho \nabla \varphi, \, \, \text{ in } \Omega \times (0,T),\\
      &   \rho' + \boldsymbol{w} \cdot \nabla \varphi = -\Delta \rho, \, \, \text{ in } \Omega \times (0,T),\\
&\mathrm{div}~\boldsymbol{w}  = 0, \, \, \text{ in } \Omega \times (0,T), \\
 &\boldsymbol{w}   = 0,  \; \frac{\partial \rho}{\partial n} = 0,  \; \frac{\partial \Delta \rho}{\partial n}  = 0  \ \text{ in } \partial \Omega \times (0, T),
     \end{aligned}
     \right.
 \end{equation}
 and the initial condition given by 
 \begin{align*}
     \boldsymbol{w}(0) = \boldsymbol{0} ,\,\, \rho(0) & = 0.
 \end{align*}
 It is easy to see, following Theorems \ref{LH weak sol} and \ref{thm4.10}, a Leray-Hopf weak solution $(\w, \rho)$ of the problem \eqref{diff wrho}  with prescribed boundary and initial conditions, exists and   satisfies the regularity given in  \eqref{exist vpsi}. Moreover, using a mollification technique used in \ref{thm4.10}, the following energy equality holds:
   \begin{align}\label{diff w}
     \|\w(t)\|^2 &+ \|\nabla\rho(t)\|^2 +2\nu\int_0^t\|\nabla\w(s)\|^2ds+ 2\int_0^t\|\Delta\rho(s)\|^2 ds\nonumber\\&+ 2\beta\int_0^t\langle|\u(s)|^{r-1}\u(s)-|\v(s)|^{r-1}\v(s),  \w(s)\rangle ds= 0.
 \end{align}  
 Arguing similarly as above, we get $\u =\v,  \, \varphi=\psi$ in $\mathbb{H}\times \mathrm{L}^2$ for all $t\in[0,T]$, that is, the system \eqref{equ} is the same as that of \eqref{equ P}. Since the weak solution of the system \eqref{equ} is unique, \eqref{diff w} shows that the weak solution of the problem \eqref{equ P} is also unique for $r\geq3$. 
\end{proof}

\begin{remark}
1. 	Note that Theorem \ref{uni} provides a result on the uniqueness of weak solutions only, whereas Theorem \ref{unique} provides a result on the continuous dependence on the initial as well as forcing data also. 

\end{remark}

 \section{Weak solution and energy equality: degenerate mobility and logarithmic potential}\label{sec5}
This section considers mobility, which degenerates at $+1, -1$, and $F$, as logarithmic potential. Let $m\in C^1([-1, 1]), \, m  \geq 0, \, \text{and } m(s) = 0 \text{ if and only if } s = 1, -1.$  Also assume that there exists $\epsilon_0$ such that $m$ is non-increasing in $[1-\epsilon_0,1]$ and  non-decreasing in $[-1, -1+\epsilon_0]$. We extend the function $m$ over whole $\mathbb{R}$ by taking $m(s)= 0 \text{ for } |s| > 1$. Moreover, $m \text{ and } F$ satisfy the following conditions:
\begin{assumption}
The functions $m \text{ and } F$ satisfy
\begin{enumerate}
    \item[(H1)] $F \in C^1((-1, 1)) \text{ such that } mF'' \in C([-1, 1]).$
    \item[(H2)] $F = F_1 + F_2, \, F_1$ is the singular part and $F_2$ is the regular part of $F$ such that $\|F_2\|_{C^2([-1, 1])}\leq C.$
    \item[(H3)] There exist $\epsilon_0>0$ such that $F_1''$ is non-decreasing in $[1-\epsilon_0, 1]$ and non-increasing in $[-1, -1+\epsilon_0]$. 
    \item[(H4)] $F(\varphi_0) \in \mathrm{L}^1(\Omega).$
    \item[(H5)] $F_1''(s) \geq \alpha \quad \forall s\in \mathbb{R},$ for some $\alpha>0$.
\end{enumerate}
\end{assumption}
The functions defined in \eqref{ex m} and \eqref{log pot} satisfy all the above  hypotheses.
Let us consider a function $G:(-1, 1) \rightarrow \mathbb{R}^+$ such that $G(0) = 0, \, G'(0) = 0, \, G''(s)= \frac{1}{m(s)}.$

In order to show the existence of a weak solution in this case, we will consider an approximated problem. We show that the limit of the solution of the approximate problem is a weak solution of the system \eqref{equ P} with logarithmic potential. Since $m$ vanishes at $1 \text{ and } -1$, we cannot bound the term $\nabla\mu$ in some $\mathbb{L}^p$ or Sobolev spaces (for example, see \eqref{fin aprox energy}). Therefore, to pass to the limit in the approximated problem, we must suitably redefine the notion of the weak solution so that $\mu$ does not appear in the variational formulation.
\begin{definition}\label{def deg weak sol}
   We  say that a pair $(\u, \varphi)$ is a \emph{weak solution} of the system \eqref{equ P} on the time interval [0, T) with initial condition $(\u_0, \varphi_0) \in \H \times \mathrm{L}^2$, if 
   \begin{align*}
       \u &\in \mathrm{L}^\infty(0, T; \H) \cap \mathrm{L}^2(0, T; \V_{\mathrm{div}}) \cap \mathrm{L}^{r+1}(0, T; \mathbb{L}^{r+1}_{\sigma}),\\
       \varphi &\in \mathrm{L}^\infty(0, T; \mathrm{H}^1)\cap \mathrm{L}^2(0, T; \mathrm{H}^2),\\
       \u' & \in \mathrm{L}^{\frac{4}{3}}(0, T; \V_{\mathrm{div}}') + \mathrm{L}^{\frac{r+1}{r}}(0, T; \mathbb{L}_\sigma^{\frac{r+1}{r}}),\quad \text{for } 1\leq r < 3,\\
       \u'  & \in \mathrm{L}^2(0, T; \V_{\mathrm{div}}') + \mathrm{L}^{\frac{r+1}{r}}(0, T; \mathbb{L}_\sigma^{\frac{r+1}{r}}), \quad \text{for } r\geq 3\\
       \varphi' & \in \mathrm{L}^2(0, T, (\mathrm{H}^1)'), \\
       \varphi & \in \mathrm{L}^\infty(\Omega\times(0, T))  \text{ and } |\varphi(x, t)| \leq 1 \quad \text{a.e.} \ (x, t) \in \Omega\times(0, T),
       \end{align*}
     where, $(\u, \varphi)$ satisfies the variational formulation
   \begin{align}
     &\int_0^T\langle \boldsymbol{u}'(s), \v(s)\rangle ds + \nu \int_0^T(\nabla\boldsymbol{u}(s), \nabla\v(s)) ds+ \int_0^T\langle (\boldsymbol{u}(s)\cdot \nabla) \boldsymbol{u}(s), \v(s) \rangle ds  \no\\
     &+ \beta \langle |\u(s)|^{r-1}\u(s), \v(s) \rangle ds  + \int_0^T(\Delta\varphi(s) \nabla \varphi(s), \v(s)) ds = \int_0^T\langle\mathbb{U}(s), \v(s)\rangle ds, \label{1weak form u}\\
     & \int_0^T\langle \varphi'(s), \psi(s)\rangle ds  + \int_0^T(\boldsymbol{u}(s) \cdot \nabla \varphi(s), \psi(s)) ds + \int_0^T(m(\varphi)\nabla \Delta\varphi(s), \nabla\psi(s)) ds  \no\\
      &+ \int_0^T(m(\varphi)F''(\varphi(s))\nabla\varphi(s), \nabla\psi(s)) ds = 0, \label{1weak form phi}
   \end{align} 
      for a.e. $s \in [0, T]$ and $\v \in \mathbb{D}_{\sigma}(\Omega \times [0, T]), \, \psi \in C^\infty(\overline{\Omega}\times [0, T])$. Additionally, it holds that 
   \begin{align}\label{deg_in_data_cgts}
       \lim\limits_{t\to 0^+}\|\u(t)-\u_0\|=0\ \text{ and }\ \lim\limits_{t\to 0^+}\|\varphi(t)-\varphi_0\|_{\mathrm{H}^1}=0.
   \end{align}
\end{definition}

\begin{theorem}\label{thm 6.3}
    Let $\Omega$  be a bounded domain in $\mathbb{R}^3$ with sufficiently smooth boundary $\partial\Omega$. Let $\varphi_0\in \mathrm{H}^1$ with $|\varphi_0| \leq 1$ a.e. and $\u_0 \in \H$. Also, let $F$ and $m$ satisfy the Hypothesis (H1)-(H5). Then there exists a weak solution of the system \eqref{equ P} in the sense of Definition \ref{def deg weak sol} and satisfies the following energy inequality:
    \begin{align}\label{deg energy}
&\frac{1}{2}\Big(\|\u(t)\|^2 + \|\varphi(t)\|^2\Big) + \int_0^t\|\sqrt{m(\varphi(s))F''(\varphi(s))}\nabla\varphi(s)\|^2 ds +\nu\int_0^t\|\nabla\u(s)\|^2 ds \no\\ &\quad  + \beta\int_0^t\|\u(s)\|^{r+1}_{\mathbb{L}_{\sigma}^{r+1}} ds
    +\int_0^t\int_\Omega\Delta\varphi(x,s)\nabla\varphi(x,s)\cdot\u(x, s) dx ds \no\\ &\leq \frac{1}{2}\Big(\|\u_0\|^2 + \|\varphi_0\|^2 \no\Big)+ \int_0^t\langle\mathbb{U}(s), \u(s)\rangle \, ds \no\\ & \quad + \int_0^t\int_\Omega m(\varphi(x,s))\nabla(\Delta\varphi(x,s))\cdot\nabla\varphi(x, s)  dx ds,
\end{align}
for all $t\in[0,T]$.
\end{theorem}
\begin{remark}
  Note that in this case also, we can argue similarly in the proofs of Corollary \ref{cor1} and Corollary \ref{cor4.4}  so that both of them hold. 
 Moreover, weak solutions are weakly continuous in time on $\H\times\mathrm{L}^2$.
\end{remark}
\begin{proof}[Proof of Theorem \ref{thm 6.3}:]
We first approximate degenerate mobility and logarithmic potential by non-degenerate mobility and regular potential, respectively. Thus, we introduce a positive non-degenerate mobility $m_\epsilon$ as 
\begin{align*}
    m_\epsilon(s) =\left\{\begin{array}{ll} m(-1+\epsilon) & \text{ for }\  s \leq -1+\epsilon,\\
                    m(s) & \text{ for }\  |s| < 1-\epsilon,\\
                    m(1-\epsilon) & \text{ for }\  s \geq 1-\epsilon,\end{array}\right.
\end{align*}
accordingly, we define $G''_\epsilon(s) = \frac{1}{m_\epsilon(s)}$ with $G_\epsilon'(0) = 0, \, G_\epsilon(0)=0.$ Note that $m_\epsilon(s) = m(s) \text{ for } |s| \leq 1-\epsilon.$ Now we write approximation of $F$ as $F_\epsilon = F_{1\epsilon} + F_2,$ where $F_{1\epsilon}''$ is given by 
\begin{align*}
    F''_{1\epsilon}(s) = \left\{\begin{array}{ll} F''_1(-1+\epsilon) & \text{ for } \ s \leq -1+\epsilon,\\
                    F_1''(s) & \text{ for } \ |s| < 1-\epsilon,\\
                    F_1''(1-\epsilon) & \text{ for }\  s \geq 1-\epsilon,\end{array}\right.
\end{align*}
with $F_{1\epsilon}(0) = F_1(0), \, F'_{1\epsilon}(0) = F_1'(0).$ From (H2), we have $\|F_2\|_{C^2([-1, 1])} \leq C$, but we can extend $F_2$ continuously to $\mathbb{R}$ so that $\|F_2\|_{C^2(\mathbb{R})} \leq C_1$.
From (H3), we have $F_{1\epsilon}(s) \leq F_1(s) \quad \forall s \in (-1, 1) \text{ and } \epsilon \in (0, \epsilon_0]$, the same $\epsilon_0$ as mentioned in (H3).  Indeed, if $1-\epsilon \leq s < 1$, and $\epsilon \leq \epsilon_0$
\begin{align*}
    F_1(s) =& F_1(1-\epsilon)+F_1'(1-\epsilon)(s-(1-\epsilon)) + \frac{1}{2}F_1''(\xi_s)(s-(1-\epsilon))^2\\
           \geq &  F_1(1-\epsilon)+F_1'(1-\epsilon)(s-(1-\epsilon)) + \frac{1}{2}F_1''(1-\epsilon)(s-(1-\epsilon))^2 =F_{1\epsilon},
\end{align*}
where $1-\epsilon < \xi_s < s.$ Similarly for $-1<s\leq -1+\epsilon$ and $\epsilon<\epsilon_0,$ we have $F_{1\epsilon}(s) \leq F_1(s)$. Further we have $F_{1\epsilon}(s) = F_1(s)$ for $|s|\leq 1-\epsilon$. We refer to \cite{2015Frigeri} for more details on these approximations.

As $\|F_2\|_{C^2(\mathbb{R})} \leq C$, there exist two positive constant $C_1$ and $C_2$ such that $|F_2(s)| \leq C_1|s|^2 + C_2.$
Thus using the assumption on initial data $\varphi_0$ and (H4), we have 
\begin{align}\label{6.4}
   \int_\Omega F_\epsilon(\varphi_0) = \int_\Omega (F_{1\epsilon} +F_2)(\varphi_0) \leq \int_\Omega (F_1(\varphi_0) + C_1|\varphi_0|^2 + C_2)  < \infty,
\end{align}
which implies $F_\epsilon(\varphi_0) \in \mathrm{L}^1(\Omega).$ Also, using (H5) and Assumption \ref{prop of F} (2), one can easily verify that
\begin{align*}
    F''_\epsilon(s) \geq -C_0 \quad \forall s\in \mathbb{R}.
\end{align*}
Then, using Theorem \ref{LH weak sol}, there exists a  pair $(\u_\epsilon, \varphi_\epsilon)$ which is a Leray-Hopf weak solution of the following approximated problem:
\begin{equation}\label{approx P}
\left\{
\begin{aligned}
&\boldsymbol{u}'_\epsilon - \nu \Delta\boldsymbol{u}_\epsilon + (\boldsymbol{u}_\epsilon\cdot \nabla)\boldsymbol{u}_\epsilon + \beta|\u_\epsilon|^{r-1}\u_\epsilon + \nabla \pi_\epsilon = \mu_\epsilon \nabla \varphi_\epsilon, \, \, \text{ in } \Omega \times (0,T), \\
&  \varphi'_\epsilon + \boldsymbol{u}_\epsilon \cdot \nabla \varphi_\epsilon = \mathrm{div \ }(m_\epsilon(\varphi_\epsilon)\nabla \mu_\epsilon), \, \, \text{ in } \Omega \times (0,T), \\
 &       \mu_\epsilon = -\Delta\varphi_\epsilon + F'_\epsilon(\varphi_\epsilon), \\
   &     \mathrm{div}~\boldsymbol{u}_\epsilon  = 0, \, \, \text{ in } \Omega \times (0,T), \\
   &   \frac{\partial\varphi_\epsilon}{\partial\n} = 0, \,  \frac{\partial\mu_\epsilon}{\partial\boldsymbol{n}}  = 0, \,\, \text{ on } \partial\Omega\times [0,T], \\
    &    \u_\epsilon  = \boldsymbol{0}, \,\, \text{ on } \partial\Omega\times [0,T], \\
     &   \boldsymbol{u}_\epsilon(0) = \boldsymbol{u}_0 ,\,\, \varphi_\epsilon(0)  = \varphi_0, \,\, \text{ in } \Omega.  
\end{aligned}   
\right.
\end{equation}
The approximated problem \eqref{approx P} is obtained by replacing logarithmic potential $F$ by its regular approximation $F_\epsilon$ and degenerate mobility $m$ by non-degenerate mobility $m_\epsilon$ in \eqref{equ P}. Therefore the approximated solution $(\u_\epsilon, \varphi_\epsilon)$ satisfies the following energy estimate:
\begin{align}\label{approx energy estimate}
   \frac{1}{2}\bigg( \|\u_\epsilon(t)\|^2 +& \|\nabla\varphi_\epsilon(t)\|^2 + \int_\Omega F_\epsilon(\varphi_\epsilon)\bigg) + \nu \int_{0}^{t}\|\nabla\u_\epsilon(s)\|^2ds + \beta \int_{0}^{t}\|\u_\epsilon(s)\|_{\mathbb{L}^{r+1}_{\sigma}}^{r + 1}ds\no\\& + \int_{0}^{t}\|\sqrt{m_\epsilon(\varphi_\epsilon)}\nabla\mu_\epsilon\|^2ds \leq  \frac{1}{2}\bigg(\|\u_0\|^2 + \|\nabla\varphi_0\|^2 + \int_\Omega F(\varphi)\bigg), 
\end{align}
for all $t\in [0, T],$ and the right-hand side is independent of $\epsilon$ (see \eqref{6.4}). Now, we will derive some a priori estimates to pass to the limit in the following variational formulation of the approximated equation:
\begin{align}
    &\langle\boldsymbol{u}'_\epsilon(t), \v(t)\rangle + \nu (\nabla\boldsymbol{u}_\epsilon(t), \nabla\v(t)) + \langle(\boldsymbol{u}_\epsilon(t)\cdot \nabla)\boldsymbol{u}_\epsilon(t), \v(t)\rangle + \beta\langle|\u_\epsilon(t)|^{r-1}\u_\epsilon(t), \v(t)\rangle \no\\ & \qquad \qquad\qquad= (-\Delta\varphi_\epsilon(t) \nabla\varphi_\epsilon(t), \v(t)),\label{ueps}\\
   & \langle\varphi_\epsilon'(t), \psi(t)\rangle + (\boldsymbol{u}_\epsilon(t) \cdot \nabla \varphi_\epsilon(t), \psi(t))  = (m_\epsilon(\varphi_\epsilon(t))\nabla(\Delta\varphi_\epsilon(t)), \nabla\psi(t)) \no\\& \qquad\qquad-(m_\epsilon (\varphi_\epsilon(t) )F''_\epsilon(\varphi_\epsilon(t))\nabla\varphi_\epsilon(t), \nabla\psi(t)\rangle,\label{phieps}
\end{align}
 for a.e. $t\in [0, T]$ and for $\v \in \mathbb{D}_{\sigma}(\Omega \times [0, T]), \, \psi \in C^\infty(\overline{\Omega}\times [0, T])$. 
We have 
\begin{align*}
    \Delta\varphi_\epsilon = -\mu_\epsilon + F'_\epsilon(\varphi_\epsilon).
\end{align*}
From the elliptic regularity theorem (also see Remark \ref{rem4.7}, Assumption \ref{prop of F} (2)), we can deduce that $\varphi_\epsilon \in \mathrm{L}^2(0, T; \mathrm{H}^2).$ Indeed we have (see \eqref{aprox del phi})
\begin{align}\label{delphi}
    \|\Delta\varphi_\epsilon\|^2 \leq \left(C_3 + \frac{1}{2}\right)\|\nabla\varphi_\epsilon\|^2 + \frac{1}{2}\|\nabla\mu_\epsilon\|^2.
\end{align}
Now we will show that $\nabla\Delta\varphi_\epsilon \in \mathrm{L}^2(0, T;\mathbb{L}^2).$ For that, we write
\begin{align}
    \nabla\Delta\varphi_\epsilon = -\nabla\mu_\epsilon + F_\epsilon''(\varphi_\epsilon)\nabla\varphi_\epsilon. \label{H3 estimate}
\end{align}
Since right hand side of \eqref{H3 estimate} is in $\mathrm{L}^2(0, T;\mathbb{L}^2)$ (see Assumption \ref{prop of F} (5)), we infer that  $\|\nabla\Delta\varphi_\epsilon\|_{\mathrm{L}^2(0,T;\mathbb{L}^2)}$ is uniformly bounded.
\vskip 0.1cm
\noindent 
\textbf{(Time derivative estimate.)} For $1\leq r < 3,$
in the variational formulation $\eqref{ueps}$ we take  $\v \in \mathrm{L}^4(0, T;\V_{\mathrm{div}})\cap \mathrm{L}^{r+1}(0, T; \mathbb{L}^{r+1}_\sigma)$ (after using a density argument)  and write
\begin{align}
    \int_0^T\langle\u'_\epsilon(t), \v\rangle dt&=  - \nu \int_0^T(\nabla\boldsymbol{u}_\epsilon(t), \nabla\v ) dt+ \int_0^T\langle(\boldsymbol{u}_\epsilon(t)\cdot \nabla)\boldsymbol{u}_\epsilon(t), \v\rangle dt \no\\& \quad+ \beta\int_0^T\langle|\u_\epsilon(t)|^{r-1}\u_\epsilon(t), \v\rangle dt + \int_0^T(\mu_\epsilon(t) \nabla \varphi_\epsilon(t), \v) dt.
\end{align}
Note that the  term involving  $\mu_\epsilon$ needs to be estimated since the estimates for  other terms will follow similarly to the non-degenerate mobility case. 
By performing an integration by parts 
\begin{align*}
    \int_0^T|(\mu_\epsilon(t)\nabla \varphi_\epsilon(t), \v)| dt=& \int_0^T|((-\Delta\varphi_\epsilon(t) +  F'_\epsilon(\varphi_\epsilon(t)))\nabla\varphi_\epsilon(t), \v)|dt\no\\
    \leq & \int_0^T\bigg|\sum_{i, j =1}^3\int_\Omega\frac{\partial^2\varphi_\epsilon}{\partial x_j^2}(t)\bigg(\frac{\partial\varphi_\epsilon}{\partial x_i}(t)\cdot v_i\bigg)\bigg| dt + \int_0^T|(\nabla F_\epsilon(\varphi_\epsilon(t)), \v)| dt\no\\
    \leq & \int_0^T \|\nabla\Delta\varphi_\epsilon(t)\|\|\v\|_{\mathbb{L}^3_{\sigma}}\|\varphi_\epsilon(t)\|_{\mathrm{L}^6} dt\no\\
    \leq & \|\varphi_\epsilon\|_{\mathrm{L}^{\infty}(0, T;\mathrm{H}^1)} \|\nabla\Delta\varphi_\epsilon\|_{\mathrm{L}^2(0,T;\mathbb{L}^2)} \|\v\|_{\mathrm{L}^2(0, T; \V_{\mathrm{div}})}.
    \end{align*}
    For $r\geq 3$ we take $\v \in \mathrm{L}^2(0, T;\V_{\mathrm{div}})\cap \mathrm{L}^{r+1}(0, T; \mathbb{L}^{r+1}_\sigma)$ and, the same estimate holds. 
Similarly, taking  $\psi \in \mathrm{L}^2(0, T;\mathrm{H}^1)$  in the equation \eqref{phieps} and using the energy estimate \eqref{approx energy estimate}, one can easily estimate $\varphi'_\epsilon$. Therefore, we have 
\begin{equation}
\left\{
    \begin{aligned}
    &\|\u'_\epsilon\|_{\mathrm{L}^\frac{4}{3}(0, T; \V_{\mathrm{div}}')} < \infty, \quad 1\leq r < 3,\\
    &\|\u'_\epsilon\|_{\mathrm{L}^2(0, T; \V_{\mathrm{div}}') + \mathrm{L}^{\frac{r+1}{r}}(0, T; \mathbb{L}_{\sigma}^{\frac{r+1}{r}})} < \infty, \quad r \geq 3,\\
    &\|\varphi'_\epsilon\|_{\mathrm{L}^2(0, T; (\mathrm{H}^1)')} < \infty.
\end{aligned}
\right.
\end{equation}
Using these uniform bounds on $\u_\epsilon$, we can extract a subsequence so that there exists  a function $\u$ such that
\begin{equation}\label{convergence}
\left\{
    \begin{aligned}
      \u_\epsilon & \xrightharpoonup* \u \text{ in } \mathrm{L}^\infty(0, T; \H) , \\
   \u_\epsilon & \rightharpoonup \u \text{ in } \mathrm{L}^2(0, T; \V_{\mathrm{div}}) , \\
    \u_\epsilon & \rightharpoonup \u \text{ in } \mathrm{L}^{r+1}(0, T; \mathbb{L}_\sigma^{r +1}) , \\
    |\u_\epsilon|^{r-1}\u_\epsilon & \rightharpoonup \z \text{ in } \mathrm{L}^{\frac{r+1}{r}}(0, T; \mathbb{L}_\sigma^{\frac{r +1}{r}}) ,\\
     \u'_\epsilon & \rightharpoonup \u' \left\{\begin{array}{ll}  \text{in } \mathrm{L}^\frac{4}{3}(0, T; \V_{\mathrm{div}}') &\text{ for } 1\leq r<3,\\  \text{ in } \mathrm{L}^2(0, T; \V_{\mathrm{div}}') + \mathrm{L}^{\frac{r+1}{r}}(0, T; \mathbb{L}_{\sigma}^{\frac{r+1}{r}})&\text{ for } r\geq 3.
     \end{array}\right.
 \end{aligned}
\right.
\end{equation}
Similarly using uniform bounds on $\varphi_\epsilon,$ we can extract a subsequence such that there exist $\varphi$ and following convergences hold:
\begin{equation}\label{convergences}
\left\{
    \begin{aligned}
    \varphi_\epsilon & \xrightharpoonup* \varphi \text{ in } \mathrm{L}^\infty(0, T; \mathrm{H^1}) , \\
     \varphi_\epsilon & \rightharpoonup \varphi \text{ in } \mathrm{L}^2(0, T; \mathrm{H}^2) \\
     \varphi'_\epsilon & \rightharpoonup \varphi' \text{ in } \mathrm{L}^2(0, T; (\mathrm{H}^1)')\\
   \varphi_\epsilon & \rightarrow \varphi \text{ in } \mathrm{L}^2(0, T; \mathrm{H^1}) ,\\
   \varphi_\epsilon(x,t) & \rightarrow \varphi(x, t)\  \text{a.e. } (x,t) \in \Omega\times(0, T),\\
   F(\varphi_\epsilon) & \xrightharpoonup* F^\ast \text{ in } \mathrm{L}^\infty(0, T; \mathrm{L}^1).\\
    \end{aligned}
\right.
\end{equation}
As discussed in \eqref{u strong}, we have  strong convergence of $\u_\epsilon$ along a subsequence still denoted by $\u_\epsilon$ for convenience.
\begin{align}
     \u_\epsilon \rightarrow \u \ \text{ in } \ \mathrm{L}^2(0, T; \H), \  \text{ and } \ \u_\epsilon(x,t)  \rightarrow \u(x,t)\ \text{ a.e. }\ (x,t)\in\Omega\times(0, T).\label{u strong1}
\end{align}
Following as in \eqref{absorption cgts} we have $\z = |\u|^{r-1}\u$. Next, we will show that $|\varphi| \leq 1$ for a.e.  $\Omega\times(0, T),$ where $\varphi$ is the limit of $\varphi_\epsilon$ as $\epsilon \rightarrow 0.$ This can be shown in a similar way as in \cite{1996_ch}. For the sake of completeness, we provide the details here.
We take the inner product by $G'_\epsilon(\varphi_\epsilon)$ in the equation $\eqref{approx P}_2$ to get 
(note that at the approximation level $G_\epsilon\in \mathrm{C}^2([0, T];\mathrm{H}^1)$)
\begin{align*}
    \frac{d}{dt}\int_\Omega G_\epsilon(\varphi_\epsilon) = -\int_\Omega m_\epsilon(\varphi_\epsilon) \nabla(-\Delta\varphi_\epsilon + F'_\epsilon(\varphi_\epsilon))G''_\epsilon(\varphi_\epsilon)\nabla\varphi_\epsilon.
\end{align*}
Since $m_\epsilon G''_\epsilon = 1,$ we have from the above expression
\begin{align*}
   \int_\Omega G_\epsilon(\varphi_\epsilon(t)) = \int_\Omega G_\epsilon(\varphi_\epsilon(0))-\int_0^t\|\Delta\varphi_\epsilon(s)\|^2 ds- \int_0^t F''_\epsilon(\varphi_\epsilon(s))\|\nabla\varphi_\epsilon(s)\|^2ds,
\end{align*}
for all $t\in[0,T]$. It follows that
\begin{align*}
  &\int_\Omega G_\epsilon(\varphi_\epsilon(t)) + \int_0^t\|\Delta\varphi_\epsilon(t)\|^2 dt + \int_0^t F''_{1\epsilon}(\varphi_\epsilon(t))\|\nabla\varphi_\epsilon(t)\|^2 \nonumber\\&\leq \int_\Omega G_\epsilon(\varphi(0)) - \int_0^t F''_2(\varphi_\epsilon(t))\|\nabla\varphi_\epsilon(t)\|^2 dt.
\end{align*}
As $F''_2\geq -C_3$, and using assumption on $m,$ we obtain $G_\epsilon(\varphi_0) \leq G(\varphi_0)$ for $\epsilon$ small enough. Applying Gronwall's inequality, we have for all $t\in[0, T]$
\begin{align}\label{6.15}
    \int_\Omega G_\epsilon(\varphi_\epsilon(t)) \leq C,
\end{align}
 Let us first assume that $s>1, \, 0<\epsilon < 1$. Then the Taylor series expansion gives
\begin{align*}
    G_\epsilon(s) &= G_\epsilon(1-\epsilon) + G_\epsilon'(1-\epsilon)(s-(1-\epsilon)) + \frac{1}{2}G''_\epsilon(1-\epsilon)(s-(1-\epsilon))^2\\
    & \geq \frac{1}{2m_\epsilon(1-\epsilon)}(s-1)^2. 
\end{align*}
Similarly, if $s < -1,$ we have 
\begin{align*}
    G_\epsilon(s) \geq \frac{1}{2m_\epsilon(\epsilon-1)}(s+1)^2.
\end{align*}
We define $\varphi_+=\max\{\varphi,0\}$. Thus, it is immediate that
\begin{align*}
    \int_\Omega(|\varphi_\epsilon|-1)^2_{+} =& \int_{\{\varphi_\epsilon>1\}}(|\varphi_\epsilon|-1)^2 + \int_{\{\varphi_\epsilon<-1\}}(|\varphi_\epsilon|-1)^2\\
     =& \int_{\{\varphi_\epsilon>1\}}(\varphi_\epsilon-1)^2 + \int_{\{\varphi_\epsilon<-1\}}(\varphi_\epsilon + 1)^2\\
    \leq & \, 2m_\epsilon(1-\epsilon)\int_{\{\varphi_\epsilon>1\}}G_\epsilon(\varphi_\epsilon) + 2m_\epsilon(\epsilon-1)\int_{\{\varphi_\epsilon<-1\}}G_\epsilon(\varphi_\epsilon)\\
    \leq & \, 2 \max \big\{m_\epsilon(1-\epsilon), m_\epsilon(\epsilon-1)\big\}\int_\Omega G_\epsilon(\varphi_\epsilon).
\end{align*}
Using \eqref{6.15}, as $\epsilon \rightarrow 0$, the right-hand side of the above inequality goes to $0$. Therefore, we obtain
\begin{align}\label{x}
    \lim_{\epsilon \rightarrow 0}\int_\Omega(|\varphi_\epsilon|-1)^2_{+} = 0.
\end{align}
Now using the convergence in \eqref{convergences} and  General LDCT (see \cite[Chapter 4, Theorem 19]{royden}), we obtain from \eqref{x} that
\begin{align*}
    \int_\Omega(|\varphi|-1)^2_{+} = 0 \ \text{ a.e. }\  t \in (0, T).
    \end{align*}
Thus, we have proved that $|\varphi(x, t) \leq 1$ a.e. in $(x, t)\in \Omega\times(0, T).$ Now, we are in a position to pass the limit in the variational formulation of the approximated problem \eqref{approx P}. It is easy to see that the variational formulation of \eqref{approx P} is given in \eqref{ueps}-\eqref{phieps}.

We need to multiply the approximate equation of $\u_\epsilon$ by $\eta \in C_c^\infty(0, T)$. Then, passing to the limit in the left hand side of \eqref{ueps} is easy and can be done in a similar way as in the case of non-degenerate mobility explained earlier. So, we will only show the convergence of the right hand side. We consider 
\begin{align}
&\Big|\int_0^T (\Delta\varphi_\epsilon(s) \nabla\varphi_\epsilon(s) - \Delta\varphi(s) \nabla\varphi(s), \eta(s)\v) ds\Big| \no\\&\leq \Big|\iint_{\Omega\times(0, T)}(\Delta\varphi_\epsilon(s) - \Delta\varphi(s))\nabla\varphi_\epsilon(s)\cdot\eta(s)\v dx ds \Big| + \Big|\int_0^T(\Delta\varphi(s)\nabla(\varphi_\epsilon(s)-\varphi(s)), \eta(s)\v) ds \Big|\no\\
&\leq \Big|\iint_{\Omega\times(0, T)}(\Delta\varphi_\epsilon(s) - \Delta\varphi(s))\nabla\varphi_\epsilon(s)\cdot\eta(s)\v dx ds \Big| \no\\& \quad
  + \int_0^T\|\Delta\varphi(s)\|_{\mathrm{L}^2}\|\nabla(\varphi_\epsilon-\varphi)(s)\|_{\mathbb{L}^3}\|\eta(s)\v\|_{\mathbb{L}^3_{\sigma}} ds\label{lim ueps}.
\end{align}
As, $\V_{\mathrm{div}}\hookrightarrow \mathbb{L}^6$ in $3$D and $\eta \in C_c^\infty(0, T)$, therefore $\eta\v \in \mathrm{L}^\infty(0, T; \V_{\mathrm{div}})$.  From \eqref{H3 estimate}, we obtain $\nabla\Delta\varphi_\epsilon \rightharpoonup \nabla\Delta\varphi$ in $\mathrm{L}^2(0, T;\mathbb{L}^2)$, which implies that the first term of \eqref{lim ueps} tend to $0$ as $\epsilon \rightarrow 0.$
Note that from the Aubin-Lions lemma, we have $\varphi_\epsilon \rightarrow \varphi$ strongly in $\mathrm{L}^2(0, T; \mathrm{H}^{\frac{3}{2}})$. Again by the Sobolev embedding, we get $\|\nabla\varphi\|_{\mathbb{L}^3} \leq C\|\varphi\|_{\mathrm{H}^{\frac{3}{2}}}$.  By an application of Gagliardo-Nirenberg's and H\"older's inequalities, we also have 
\begin{align*}
&	\int_0^T\|\Delta\varphi(s)\|_{\mathrm{L}^2}\|\nabla\varphi(s)\|_{\mathbb{L}^3}\|\eta(s)\v\|_{\mathbb{L}^3_{\sigma}} ds\nonumber\\&\leq C\left(\int_0^T\|\Delta\varphi(s)\|_{\mathrm{L}^2}^2ds\right)^{3/4}\sup_{t\in[0,T]}\|\nabla\varphi(s)\|\sup_{t\in[0,T]}\|\v(s)\|^{1/2}\left(\int_0^T\|\nabla\v(s)\|^2ds\right)^{1/4}<\infty.
\end{align*}
 Thus the limit of the second term in the  expression \eqref{lim ueps} goes to $0$ as $\epsilon \rightarrow 0.$ Similarly we multiply the equation \eqref{phieps} by $\chi\in C_c^\infty(0, T)$ and integrate from $0$ to $T$ yields
\begin{align*}
   \lim_{\epsilon \rightarrow 0} \int_0^T\langle\varphi_\epsilon'(t), \chi(t)\psi\rangle dt = -\lim_{\epsilon \rightarrow 0}\int_0^T (\boldsymbol{u}_\epsilon(t) \cdot \nabla \varphi_\epsilon(t), \chi(t)\psi) dt + \int_0^T( m^\ast(t), \chi(t)\nabla\psi) dt,
\end{align*}
where $m^*$ is the weak limit of $m_\epsilon(\varphi_\epsilon(t))\nabla(\Delta\varphi_\epsilon(t)) -(m_\epsilon (\varphi_\epsilon(t) )F''_\epsilon(\varphi_\epsilon(t))\nabla\varphi_\epsilon(t)$.  Our aim is to show that $m^*=m(\varphi)\nabla(\Delta\varphi)-m(\varphi)F''(\varphi)\nabla\varphi$. Using the convergences given in \eqref{convergences}, one can easily see that 
\begin{align*}
    \lim_{\epsilon \rightarrow 0} \int_0^T\langle\varphi_\epsilon'(t), \chi(t)\psi\rangle dt=&  \int_0^T\langle\varphi'(t), \chi(t)\psi\rangle dt,\\
    \lim_{\epsilon \rightarrow 0}\int_0^T(\boldsymbol{u}_\epsilon(t) \cdot \nabla \varphi_\epsilon(t), \chi(t)\psi) dt =& \int_0^T( \boldsymbol{u}(t) \cdot \nabla \varphi(t), \chi(t)\psi) dt .
\end{align*}
Let us take $\psi \in C^1(\overline{\Omega})$ and $\boldsymbol{\omega}(t, x) = \chi(t)\nabla\psi(x) \in C(\overline{\Omega}\times[0, T] )$. To find $m^\ast$, we want to pass to the limit in the following:
\begin{align}\label{rhs limit}
\int_{\Omega\times(0, T)}m_\epsilon(\varphi_\epsilon)\nabla\mu_\epsilon \cdot \boldsymbol{\omega} = \int_{\Omega\times(0, T)}m_\epsilon(\varphi_\epsilon)\nabla(-\Delta\varphi_\epsilon + F_\epsilon'(\varphi_\epsilon))\cdot\boldsymbol{\omega}.
\end{align}
We can write
\begin{align}
    \int_{\Omega\times(0, T)}m_\epsilon(\varphi_\epsilon)\nabla(-\Delta\varphi_\epsilon)\cdot\boldsymbol{\omega} = & \int_{\Omega\times(0, T)} m_\epsilon(\varphi_\epsilon)\Delta\varphi_\epsilon \nabla\cdot\boldsymbol{\omega}+ \int_{\Omega\times(0, T)} m_\epsilon'(\varphi_\epsilon)\Delta\varphi_\epsilon\nabla\varphi_\epsilon\cdot\boldsymbol{\omega}  \no\\
    =: & A + B.
\end{align}
Note that for all $s\in\mathbb{R},$ we have 
\begin{align*}
    |m_\epsilon(s) - m(s)| \leq \sup_{-1\leq s \leq -1+\epsilon}|m(-1+\epsilon)-m(s)| + \sup_{1-\epsilon\leq s \leq 1}|m_\epsilon(1-\epsilon) - m(s)| \rightarrow 0 \text{ as } \epsilon \rightarrow 0.
\end{align*}
Therefore $m_\epsilon \rightarrow m$ uniformly. Thus from the pointwise convergence of $\varphi_\epsilon$, we have
\begin{align}\label{uniform m cgts}
    m_\epsilon(\varphi_\epsilon) \rightarrow m(\varphi)\  \text{ a.e. in  }\ \Omega\times(0, T).
\end{align}
Now, we are in a position to pass to the limit in $A$. Let us consider 
\begin{align}
    &\Big|\int_{\Omega\times(0, T)} m_\epsilon(\varphi_\epsilon)\Delta\varphi_\epsilon \nabla\cdot\boldsymbol{\omega} - \int_{\Omega\times(0, T)} m(\varphi)\Delta\varphi \nabla\cdot\boldsymbol{\omega}\Big|\no\\
    & \leq \esssup_{(x,t)\in \Omega\times(0, T)}|m_\epsilon(\varphi_\epsilon)-m(\varphi)|\|\varphi\|_{\mathrm{L}^2(0, T;\mathrm{H}^2)}\|\boldsymbol{\omega}\|_{\mathrm{L}^2(0, T; \H^1)}+\left| \iint_{\Omega\times(0, T)}m(\varphi)(\Delta\varphi_\epsilon-\Delta\varphi)\nabla\cdot\boldsymbol{\omega}\right|\no\\
    & \rightarrow 0 \text{ as } \epsilon \rightarrow 0.\label{lim A}
\end{align}
As in \eqref{uniform m cgts}, we can show that $m'_\epsilon \rightarrow m'$ uniformly which gives
\begin{align}\label{uniform m' cgts}
 m'_\epsilon(\varphi_\epsilon) \rightarrow m'(\varphi) \ \text{ a.e. in  }\ \Omega\times(0, T).
\end{align}
Also, from the strong convergence of $\varphi$ in \eqref{convergences}, we have $\nabla\varphi_\epsilon \rightarrow \nabla\varphi$ in $\mathrm{L}^2(0, T;\mathbb{L}^2)$ and $\nabla\varphi_\epsilon \rightarrow \nabla\varphi$ a.e. in $\Omega\times(0, T)$ (along a subsequence). Thus, we have 
\begin{align}\label{mphi product cgts}
m'_\epsilon(\varphi_\epsilon)\nabla\varphi_\epsilon \rightarrow m'(\varphi)\nabla\varphi \ \text{ in } \mbox{a.e. in $\Omega\times(0, T)$.}
\end{align}
Now we calculate the limit of $B$ as follows
\begin{align}
   & \Big|\int_{\Omega\times(0, T)} \Delta\varphi_\epsilon m_\epsilon'(\varphi_\epsilon)\nabla\varphi_\epsilon\cdot\boldsymbol{\omega} - \int_{\Omega\times(0, T)}\Delta\varphi m'(\varphi)\nabla\varphi\cdot\boldsymbol{\omega}\Big|\no\\
    &\leq \|\Delta\varphi_\epsilon\|_{\mathrm{L}^2(0, T; \mathrm{L}^2)}\|m'_\epsilon(\varphi_\epsilon)\nabla\varphi_\epsilon - m'(\varphi)\nabla\varphi\|_{\mathrm{L}^2(\Omega\times(0, T)}\|\boldsymbol{\omega}\|_{\mathrm{L}^\infty(\Omega\times(0, T))} \no\\
    &\quad+\left| \iint_{\Omega\times(0, T)}m'(\varphi)(\Delta\varphi_\epsilon-\Delta\varphi)\nabla\varphi\cdot\boldsymbol{\omega}\right|\rightarrow 0 \text{ as } \epsilon \rightarrow 0\label{lim B},
\end{align}
where we have used the LDCT and weak convergence of $\Delta\varphi_\epsilon$ to $\Delta\varphi$. Now we want to pass to the limit on the second term of \eqref{rhs limit}. We know that $m_\epsilon F''_\epsilon$ is uniformly bounded in $\mathbb{R}$, so it has convergent subsequence by the Bolzano-Wierstrass Theorem. We want to show that the limit of the sequence is $m(\varphi)F''(\varphi)$, that is, 
\begin{align}\label{third term convergence}
    m_\epsilon(\varphi_\epsilon)F_\epsilon''(\varphi_\epsilon) \rightarrow m(\varphi)F''(\varphi) \quad \ \text{a.e. in }\  \Omega\times(0, T).
\end{align}
If $|\varphi(x, t)| < 1$, then $F_\epsilon''(\varphi_\epsilon) = F''(\varphi_\epsilon)$. Since $F \in C^2((-1, 1)),$ we have $F_\epsilon''(\varphi_\epsilon) \rightarrow F''(\varphi)$ a.e. in $\Omega\times(0, T).$ So we have \eqref{third term convergence} in this case by an application of \eqref{uniform m cgts}.

Now let us consider points the $(x, t)$ such that $|\varphi(x, t)| = 1$. Without loss of generality, we can take $\varphi_\epsilon(x, t) \rightarrow 1 = \varphi(x, t).$ If $\varphi_\epsilon(x, t) \geq 1 -\epsilon,$ then $m_\epsilon(\varphi_\epsilon)F''_\epsilon(\varphi_\epsilon) \rightarrow m(1)F''(1) = m(\varphi(x, t))F''(\varphi(x, t))$. Similarly we can conclude, if $\varphi_\epsilon(x, t) \leq 1 -\epsilon$. Let us now consider 
\begin{align}
    &\Big|\int_{\Omega\times(0, T)} m_\epsilon(\varphi_\epsilon)\nabla F_\epsilon'(\varphi_\epsilon)\cdot\boldsymbol{\omega} - \int_{\Omega\times(0, T)} m(\varphi)\nabla F'(\varphi)\cdot\boldsymbol{\omega}\Big|\no\\
    & \leq \sup_{(x,t)\in \Omega\times(0, T)}| m_\epsilon(\varphi_\epsilon)F_\epsilon''(\varphi_\epsilon)-m(\varphi)F''(\varphi)|\|\nabla\varphi_\epsilon\|_{\mathrm{L}^2(0, T;\mathbb{L}^2)}\|\boldsymbol{\omega}\|_{\mathrm{L}^2(0, T; \mathbb{L}^2)}\no\\
    & \quad+ \sup_{(x,t)\in \Omega\times(0, T)}|m(\varphi)F''(\varphi)|\|\nabla\varphi_\epsilon-\nabla\varphi\|_{\mathrm{L}^2(0, T; \mathbb{L}^2)}\|\boldsymbol{\omega}\|_{\mathrm{L}^2(0, T; \mathbb{L}^2)}\no\\
    & \rightarrow 0 \text{ as } \epsilon \rightarrow 0.\label{last term}
\end{align}
By the uniqueness of weak limit, one can obtain $m^*=m(\varphi)\nabla(\Delta\varphi)-m(\varphi)F''(\varphi)\nabla\varphi$.

With the help of the equations \eqref{lim ueps}, \eqref{lim A}, \eqref{lim B} and \eqref{last term}, we can pass to the limit as $\epsilon \rightarrow 0$ in the equations \eqref{ueps} and \eqref{phieps} to get the variational formulation \eqref{1weak form u}, \eqref{1weak form phi}.

Now to show the uniform energy estimate, we can take test function $\psi = \varphi$ in the equation \eqref{1weak form phi} and obtain
\begin{align}
    &\frac{1}{2}\|\varphi(t)\|^2 + \int_0^t\|\sqrt{m(\varphi(s))F''(\varphi(s))}\nabla\varphi\|^2 ds\no\\&= \frac{1}{2}\|\varphi(0)\|^2 +  \int_0^t\int_\Omega m(\varphi(x,s))\nabla(\Delta\varphi(x,s))\cdot\nabla\varphi(x, s) dx ds,\label{energy half}
\end{align}
for all $t\in[0,T]$. Since the energy equality holds for $\u_{\epsilon},$ in the equation \eqref{ueps}, we take $\v = \u_\epsilon$ to get
\begin{align}
\frac{1}{2}&\|\u_\epsilon(t)\|^2 + \nu\int_0^t\|\nabla\u_\epsilon(s)\|^2 ds + \beta\int_0^t\|\u_\epsilon(s)\|^{r+1}_{\mathbb{L}_{\sigma}^{r+1}} ds\no\\
    &= \frac{1}{2}\|\u_\epsilon(0)\|^2 + \int_0^t\int_\Omega\Delta\varphi_\epsilon(x,s)\nabla\varphi_\epsilon(x,s)\cdot\u_\epsilon(x, s) dx ds,
\end{align}
for all $t\in[0,T]$. Now one can pass the limit as in the non-degenerate mobility case to obtain for all $t\in[0, T]$
\begin{align}
\frac{1}{2}&\|\u(t)\|^2 + \nu\int_0^t\|\nabla\u(s)\|^2 ds + \beta\int_0^t\|\u(s)\|^{r+1}_{\mathbb{L}_{\sigma}^{r+1}} ds\no\\
    &\leq \frac{1}{2}\|\u_0\|^2 + \int_0^t\int_\Omega\Delta\varphi(x,s)\nabla\varphi(x,s)\cdot\u(x, s) dx ds.\label{energy other half}
\end{align}
Adding \eqref{energy half} and \eqref{energy other half}, we derive our claimed energy inequality \eqref{deg energy}.
The only thing we left to prove is the convergence of the initial data. This part follows in a similar way  as in  step 8 of the proof of Theorem \ref{LH weak sol}.
\end{proof}
In the following theorem, we will show the inequality sign that appears in \eqref{deg energy} can be replaced by equality so that energy equality holds.
\begin{theorem}\label{thm 6.4}
  Every weak solution with initial data $(\u_0, \varphi_0)\in \H\times\mathrm{H}^1$ of the system \eqref{equ P} on a bounded domain satisfies the energy equality  
      \begin{align}\label{deg ee}
&\frac{1}{2}\Big(\|\u(t)\|^2 + \|\varphi(t)\|^2\Big) + \int_0^t\|\sqrt{m(\varphi(s))F''(\varphi(s))}\nabla\varphi\|^2 ds +\nu\int_0^t\|\nabla\u(s)\|^2 ds   \no\\ &\quad+ \beta\int_0^t\|\u(s)\|^{r+1}_{\mathbb{L}_{\sigma}^{r+1}} ds
   +\int_0^t(\Delta\varphi(s)\nabla\varphi(s),\u(s)) ds\\&= \frac{1}{2}\Big(\|\u_0\|^2 + \|\varphi_0\|^2 \no\Big)+ \int_0^t\langle\mathbb{U}(s), \u(s)\rangle \, ds+\int_0^t( m(\varphi(s)) \nabla\Delta\varphi(s), \nabla\varphi(s)) ds,
\end{align}
for all $t\in[0,T]$ and $r\geq 3$.
\end{theorem}
\begin{proof}
 The proof will follow in a similar way to the energy equality proof for non-degenerate mobility and regular potential (cf. Theorem \ref{thm4.10}). Therefore, we briefly discuss the proof here. Let us take the function $\u^n_h$ from \eqref{mol u}, where $(\u,\varphi)$ is a weak solution of the system \eqref{equ P} with singular potential and degenerate mobility. To get energy equality, we take $\u^n_h$ as a test function in the equation \eqref{1weak form u},
 \begin{align}\label{degeu}
     &\int_0^t\langle \boldsymbol{u}'(s), \u^n_h(s)\rangle ds + \nu \int_0^t(\nabla\boldsymbol{u}(s), \nabla\u^n_h(s)(s)) ds+ \int_0^t\langle (\boldsymbol{u}(s)\cdot \nabla) \boldsymbol{u}(s), \u^n_h(s)(s) \rangle ds  \no\\
     &+ \beta \langle |\u(s)|^{r-1}\u(s), \u^n_h) \rangle ds  + \int_0^t(\Delta\varphi(s) \nabla \varphi(s), \u^n_h(s)) ds = \int_0^t\langle\mathbb{U}(s), \u^n_h(s)\rangle ds.
     \end{align}
 Then we pass to the limit as $n\to\infty$ and $h\to 0$ in the above equation. We only need to check the limit of the last term of the left hand side of \eqref{1weak form u}; the other terms convergence will follow similar to the proof described in Theorem \ref{thm4.10}.  We consider
\begin{align*}
     \Big|\int_0^t&(\Delta\varphi(s)\nabla\varphi(s), \u^n_h(s)-\u_h(s)) ds\Big| \leq C \int_0^t\|\Delta\varphi(s)\|_{\mathrm{L}^3}\|\nabla\varphi(s)\|_{\mathbb{L}^6}\|\u^n_h(s)-\u_h(s)\|_{\H} \, ds\\
     & \leq C\|\nabla\Delta\varphi\|^\frac{1}{2}_{\mathrm{L}^2(0, T; \mathbb{L}^2)}\|\varphi\|^{\frac{3}{2}}_{\mathrm{L}^2(0, T; \mathrm{H}^2)}\|\u^n_h-\u_h\|_{\mathrm{L}^\infty(0, T; \H)} \to 0 \text{ as } n \to \infty.
 \end{align*}

 Further,
 \begin{align*}
      \Big|\int_0^t(&\Delta\varphi(s)\nabla\varphi(s), \u_h(s)-\u(s)) ds\Big| = \Big|\int_0^t(\nabla\Delta\varphi(s)\cdot (\u_h-\u)(s), \varphi(s)) ds \Big|\\
      & \leq \int_0^t\|\nabla\Delta\varphi(s)\|_{\mathrm{L}^2}\|\u_h(s)-\u(s)\|_{\mathbb{L}^3_{\sigma}}\|\varphi(s)\|_{\mathrm{L}^3} ds\\
      & C\leq \|\varphi\|_{\mathrm{L}^\infty(0, T; \mathrm{L}^2)}\|\nabla\varphi\|_{\mathrm{L}^\infty(0, T; \mathbb{L}^2)}\|\nabla\Delta\varphi\|_{\mathrm{L}^2(0, T; \mathbb{L}^2)}\|\u_h-\u\|_{\mathrm{L}^2(0, T; \V_{\mathrm{div}})} \to 0 \text{ as } h\to 0.
 \end{align*}
 Finally, we obtain from \eqref{degeu} that
    \begin{align}\label{degee u}
    \frac{1}{2}\|\u(t)\|^2 + \nu \int_0^{t} \|\nabla\u(s)\|^2 \, ds  + \beta \int_0^{t}\|\u(s)\|^{r+1}_{\mathbb{L}_{\sigma}^{r+1}(\Omega)} +\int_0^t(\Delta\varphi(s)\nabla\varphi(s), \u(s)) ds \no\\=  \frac{1}{2}\|\u(0)\|^2  +\int_0^{t_1}\langle\mathbb{U}(s), \u(s)\rangle ds,
\end{align}
for all $t \in [0, T].$ 
 For the $\varphi$ equation, we replace the test function $\psi$ by $\varphi$ in \eqref{1weak form phi} and get 
 \begin{align}\label{degee phi}
  & \frac{1}{2}\|\varphi(t)\|^2+ \int_0^t\|\sqrt{m(\varphi(s))F''(\varphi(s))}\nabla\varphi(s)\|^2 ds\nonumber \\&= \frac{1}{2}\|\varphi_0\|^2 +\int_0^t(m(\varphi(s))\nabla \Delta\varphi(s), \nabla\varphi(s)) ds,
 \end{align}
for all $t\in[0,T]$. Adding \eqref{degee u} and \eqref{degee phi}, one obtains the required result.
\end{proof}

\bibliographystyle{plain}
\bibliography{mybibliography}
\end{document}